\documentclass[reqno,11pt]{amsart}

\usepackage[T1]{fontenc}
\usepackage[utf8]{inputenc}
\usepackage[english]{babel}
\usepackage{amssymb, amsmath, color, textcomp, url,tikz}
\usepackage{tkz-tab}
\usepackage{caption}
\usepackage{float}
\usetikzlibrary{matrix,arrows,shapes}
\usetikzlibrary{tikzmark, decorations.pathreplacing}
\usetikzlibrary{fit}
\usetikzlibrary{positioning}
\usepackage[labelformat=empty]{caption}
\usepackage{mdwlist}
\usepackage{etoolbox}
\usepackage{soul}
\usepackage{xcolor} 

\usepackage{enumerate,xspace}
\usepackage{subcaption}

\theoremstyle{plain}
\newtheorem{theorem}{Theorem}[section]
\newtheorem{cor}[theorem]{Corollary}
\newtheorem{prop}[theorem]{Proposition}
\newtheorem{lemma}[theorem]{Lemma}
\newtheorem{hyp}{Property}

\newtheorem*{thA}{Theorem A}
\newtheorem*{thB}{Theorem B}

\newcounter{proofcount}
\AtBeginEnvironment{proof}{\stepcounter{proofcount}}

\newtheorem*{claim*}{Claim}
\makeatletter                  
\@addtoreset{claim}{proofcount}
\makeatother                   

\newenvironment{claimproof*}[1][Proof of Claim.] 
{%
	\proof[#1]%
	
}
{%
	\endproof%
}

\theoremstyle{definition}
\newtheorem{remark}[theorem]{Remark}

\newtheorem{definition}[theorem]{Definition}

\newtheorem*{question}{Question}

\newcommand{\nc}{\newcommand}

\nc{\Z}{\mathbb{Z}}
\nc{\N}{\mathbb{N}}
\nc{\Q}{\mathbb{Q}}

\nc\LL{\mathcal L}
\nc\UU{\mathbb U}
\nc\M{\mathcal M}
\newcommand{\fix}{{\rm Fix}}

\def\restr #1{\!\mathop{_{\upharpoonright #1}}}
\nc{\acl}{\operatorname{acl}}
\nc{\dcl}{\operatorname{dcl}}
\nc{\aclq}{\operatorname{acl^{eq}}}
\nc\SU{\operatorname{SU}}
\nc\RM{\operatorname{RM}}

\nc{\paraset}{Z}

\nc{\scl}[1]{\langle #1 \xspace \rangle}
\nc{\cl}{\operatorname{cl}}
\nc\inv{ ^{-1}}
\nc\clS{\mathcal{S}}
\nc\Aut{\operatorname{Aut}}

\nc{\tp}{\operatorname{tp}}
\nc{\qftp}{\operatorname{qftp}}
\nc{\cf}{\text{cf.}\xspace}

\def\Ind#1#2{#1\setbox0=\hbox{$#1x$}\kern\wd0\hbox to 0pt{\hss$#1\mid$\hss}
	\lower.9\ht0\hbox to 0pt{\hss$#1\smile$\hss}\kern\wd0}
\def\Notind#1#2{#1\setbox0=\hbox{$#1x$}\kern\wd0\hbox to
	0pt{\mathchardef\nn="0236\hss$#1\nn$\kern1.4\wd0\hss}\hbox
	to 0pt{\hss$#1\mid$\hss}\lower.9\ht0
	\hbox to 0pt{\hss$#1\smile$\hss}\kern\wd0}
\def\ind{\mathop{\mathpalette\Ind{}}}
\def\nind{\mathop{\mathpalette\Notind{}}}

\def\ld{\mathop{\ \ \hbox to 0pt{\hss$\mid^{\hbox to
				0pt{$\scriptstyle\mathrm{ld}$\hss}}$\hss}
		\lower4pt\hbox to 0pt{\hss$\smile$\hss}\ \ }}

\def\indi#1{\mathop{\ \
		\hbox to 0pt{\hss$\mid^{\hbox to 0pt{$\scriptstyle#1$\hss}}$\hss}
		\lower4pt\hbox to 0pt{\hss$\smile$\hss}\ \ }}
\def\nindi#1{\mathop{\ \ \hbox to
		0pt{\hss$\!\not{\mid}^{\hbox to 0pt{$\scriptstyle\,#1$\hss}}$\hss}
		\lower4pt\hbox to 0pt{\hss$\smile$\hss}\ \ }}

\begin{document}

\title[Automorphism group of fields with operators]{Simplicity of the
automorphism group of fields with operators}
\date{\today}

\author[BLOSSIER ET AL.]{Thomas Blossier, Zo\'e Chatzidakis, Charlotte
Hardouin and Amador Martin-Pizarro}

\address{Universit\'e de Lyon; CNRS; Universit\'e Lyon 1; Institut Camille
        Jordan UMR5208, 43 boulevard du 11
        novembre 1918, F--69622 Villeurbanne Cedex, France}
\email{blossier@math.univ-lyon1.fr}
\address{IMJ-PRG, Universit\'e Paris Cit\'e, 8 place Aur\'elie Nemours,
 75013 Paris, France}
\email{zoe.chatzidakis@imj-prg.fr}
	\address{
	Institut de Math\'ematiques de Toulouse UMR5219; Institut Universitaire de France;  Universit\'e
	Paul Sabatier, 118 route de Narbonne,  \newline F--31062 Toulouse Cedex 9,
	France}
\email{charlotte.hardouin@math.univ-toulouse.fr}
\address{Mathematisches Institut,
        Albert-Ludwigs-Universit\"at Freiburg, D-79104 Freiburg, Germany}
\email{pizarro@math.uni-freiburg.de}
\thanks{The first three authors thank the  support of the program EFI
GDR-2052 as well as the ANR project AGRUME ANR-17-CE40-0026.  The first two and the fourth authors were partially
        supported by the program GeoMod ANR-19-CE40-0022-01 (ANR-DFG). The third
author was partially supported by the project ANR De rerum
        natura ANR-19-CE40-0018. The fourth author was partially supported by
        the program
        PID2020-116773GB-I00.}
\keywords{Model Theory, Differentially and difference
        closed fields, Automorphism groups}
\subjclass{03C45, 12H05, 12H10}

\begin{abstract}
        We adapt a proof of Lascar in order to show the simplicity of the group
        of automorphisms fixing pointwise all non-generic elements for a class
        of uncountable models of suitable theories, encompassing both 
        strongly minimal theories as well as several theories of fields 
        with operators.
\end{abstract}

\maketitle

\begin{flushright}
\begin{minipage}{.40\textwidth}
\small
\emph{
$\dots$A lot of things unsaid as well.\\ 
We shout and argue and fight, \\ and work it on out.\\ 
The sum is greater than the parts.}\\ 
(Pink Floyd -- \emph{Things left unsaid})\\
\hfill
To Zoé, in Memoriam.  
\end{minipage}
\end{flushright}

\bigskip

\section{Introduction}
The study of the lattice of normal subgroups of the automorphism 
groups of classical structures is a recurrent topic in mathematics. 
Schreier and Ulam \cite{SU33} classified all normal subgroups of the 
permutation group $\mathrm{Sym}(\N)$ of the integers.  (The 
anonymous referee of a previous version of this article let us know 
that their result was already implicitly contained in Onofri's work 
\cite{lO29}). In order to classify all normal subgroups of the permutation group of uncountable sets,  Baer \cite[Satz]{rB34} shows that the group $\mathrm{Sym}(\kappa)$ is simple modulo the normal subgroup of those automorphisms with support strictly smaller than $\kappa$.  Rosenberg obtains a similar classification for the general linear group of a 
vector space of dimension $\kappa$ over a fixed field and shows, among others, that this group is again simple, modulo the normal subgroup of those isomorphisms which differs from a homothecy by a linear transformation of range strictly less than $\kappa$ 
\cite[Proposition 3.6 \& Lemma 3.7]{aR58} (see also \cite{HWZ22})

Infinite sets, infinite vector (or affine) spaces over a ground  field as 
well as  algebraically closed fields of some fixed characteristic 
constitute the three archetypes of strongly minimal sets.  In 
\cite{dL92} Lascar 
studied the structure of the automorphism group of a countable (almost) strongly minimal set $\mathbb M$ of infinite dimension, more 
precisely, the normal subgroup $\Aut(\mathbb 
M/\acl^{eq}(\emptyset))$ of those automorphisms fixing the 
algebraic closure (in the imaginary expansion $\mathbb M^{eq}$) of 
the empty set. He showed 
in particular that this subgroup is simple modulo a certain class of 
automorphisms, which he referred to as \emph{bounded} (see 
Definition \ref{D:bounded}). Under the continuum hypothesis,  he 
concluded from his result in the countable case the
simplicity of the group of field automorphisms of the complex numbers
$\mathbb C$ which fix pointwise the algebraic closure $\overline{\mathbb
	Q}$ of the prime field. Lascar himself wrote \cite[p. 249]{dL92} \emph{J'ai
	peine \`a imaginer que ce fait n'est pas d\'ej\`a connu} (I cannot imagine
that this result is not already well-known). Lascar's proof 
\cite{dL92} strongly used  that the group of automorphisms of
a countable structure is naturally a Polish group, and hence many of 
his
arguments have a topological flavour. His proof has two main
ingredients: firstly, the only bounded  automorphism of the field
$\mathbb C$ is the identity. Secondly, he manages to 
fuse partial
elementary maps defined on independent subsets, since algebraic 
independence  coincides with non-forking independence for the stable 
theory of the field $\mathbb C$,
which eliminates imaginaries, so types over  algebraically closed
subsets of $\mathbb C$ are stationary. Lascar's proof was adapted by
Evans, Ghader\-nezad and Tent \cite[Example 3.14]{EGT16} to 
show 
the simplicity of the group of
automorphisms of a countable saturated differentially closed field 
fixing pointwise all differentially algebraic elements. \\

In  \cite{dL97} Lascar provided a more direct proof of the simplicity 
of $\Aut(\mathbb
C/\overline{\mathbb Q})$ in pure algebraic terms, circumventing 
both the use of topological
arguments as well as the continuum hypothesis. His second proof 
holds for every uncountable almost strongly minimal set in a countable 
language (see \cite[Partie 4]{dL93}) and relies on a clever
use via transfinite induction of the two 
main  ingredients of his original proof \cite{dL92}.

Motivated by Lascar's latter approach 
in the uncountable case, we will present here a unifying 
approach to 
the proof given in \cite{dL97} (see also the version 
in French \cite{dL93}) to study  the automorphism
group of several uncountable fields equipped with operators. More 
precisely, we show the simplicity of the group of  automorphisms
fixing the closure of the prime
field with respect to a natural pregeometry, modulo the subgroup of
bounded automorphisms (see Definition \ref{D:bounded}). The description of bounded
automorphisms of several theories of fields with
operators already appears in \cite{BHMP17, fW20}. Our work
deals mainly with the case of stable theories: all but one 
example are stable. We retrieve  the existing results  (since they  are all  $\omega$-stable), and get
several new examples. Our method gives the blueprint to obtain many more.  
The unstable example is the theory ACFA$_0$
of existentially closed difference fields of characteristic $0$, which
is unstable, but however supersimple, and therefore has a good notion of
independence. It is likely that our methods will also work for
other well-behaved simple theories.

Whilst uncountably saturated models 
exist for stable theories, this need no longer be the case for simple 
theories without assuming additional set-theoretical assumptions.
A
generalisation
of
saturation in its own cardinality $\kappa>\aleph_0$ is the notion of 
$\kappa$-primeness over a given set of parameters (see Definition \ref{D:prime}). The existence 
and uniqueness of 
$\kappa$-prime models over a given subset of parameters  is due to 
Shelah for
stable theories \cite[Theorem IV.4.18]{sSbook} under certain
set-theoretic assumptions on the cardinal $\kappa$. An analogous result was obtained by the second author (\cite[Theorem 3.14]{zC23}) for the unstable theory ACFA$_0$.

In this work, we will provide a general approach, encompassing both 
the case of uncountably saturated
models (if they exist) as well as the case of  $\kappa$-prime 
models, for
suitable theories of fields with operators, to deduce the following  
results: 
\begin{thA}\textup{(}\cf~\ Theorem \ref{T:main} \& Remark \ref{R:unbounded_movemax2} \textup{)}~ 
	Consider a complete simple theory $T$ in some language $\LL$ satisfying Properties
	\ref{H:single_gen}-\ref{H:addgen}  as in Section
	\ref{S:Prel}. Fix a cardinal $\kappa\ge |\LL|^+$ as well as a $\kappa$-tame model $\UU$ over a subset $\paraset$ of size strictly less than $\kappa$ (see Definition \ref{D:tame_model}). For every subset $\paraset$, we have its corresponding closure  $\cl(\paraset)$ as in Definition \ref{D:closure}.
	
	Given an  automorphim $\tau$ of  $\UU$  fixing $\cl(\paraset)$ pointwise which moves maximally (see Definition \ref{D:moves_max}), every  automorphism $\nu$ of $\UU$ fixing $\cl(\paraset)$ pointwise can be written as the product of four conjugates of $\tau$ and $\tau\inv$.
	
	In particular, under the conditions of Remark \ref{R:unbounded_movemax}, the group $\Aut(\UU/\cl(\paraset))$ of automorphisms of
	$\UU$ fixing  $\cl(\paraset)$ pointwise is simple modulo the normal
	subgroup of all bounded automorphisms fixing  $\cl(\paraset)$ pointwise (see Definition \ref{D:bounded}).
\end{thA} 

For some particular theories (listed below), an explicit description of the bounded automorphisms can be deduced. In the particular examples of uncountable sets and vector spaces over a field of uncountable dimension (both $1$-based strongly minimal), we recover thus the results of Baer \cite{rB34} and Rosenberg \cite{aR58} mentioned at the beginning: 
\begin{thB}\textup{(}cf.~ Proposition \ref{P:classical_sm} \& Theorem \ref{T:mainExamples}\textup{)}~ 
	
	For the two classical examples of $1$-based strongly minimal theories, we have the following: Given an uncountable cardinal $\kappa$ and an infinite set $\mathcal M$ of cardinality $\kappa$ with no additional structure, the group of permutations of $\mathcal M$ is simple modulo the subgroup of permutations of support strictly less than $\kappa$.
	
	Similarly, given a field $K$ and an  uncountable cardinal $\kappa>|K|$, the group of isomorphisms of a $K$-vector space $V$ of cardinality $\kappa$ is simple modulo the subgroup of those isomorphisms $\varphi$ such that for some $\lambda$ in $K$, the eigenspace $\mathrm{Ker}(\varphi-\lambda\cdot \mathrm{Id}_V)$ has codimension strictly less than $\kappa$. 
	
	Moreover, for each of the following countable theories of fields with operators:
	\begin{itemize}
		\item algebraically closed fields with the closure operator given by the
		field algebraic closure;
		\item differentially closed fields in characteristic $0$ with finitely
		many commuting  derivations with the closure operator given by the
		elements which are not differentially transcendental;
		\item  differential fields in characteristic $0$, maximal
		with the property of omitting a given strictly minimal
		type $X$ with the same closure as above;
		\item difference  closed fields  in characteristic $0$ with the closure
		operator given by the elements of transformal transcendence degree $0$;
		
		\item  proper pairs of algebraically
		closed fields $(K,E)$ with the closure operator \[ \cl(B)= {
			E(B)^\mathrm{alg}};\]  
		\item separably closed fields $K$ in characteristic $p$ and
		infinite imperfection degree with the closure
		operator \[\cl(B)=\bigcap\limits_{n\in\N} \bigl(K^{p^n}\acl(B)\bigr).\]            
	\end{itemize}
	Given an uncountable model saturated $\UU$ in its
	uncountable cardinality $\kappa$ (if such models exist) and a subset $\paraset$  of parameters of size strictly less than $\kappa$, the group of automorphisms of $\UU$ fixing pointwise $\cl(\paraset)$ is simple.

	More generally, for any of the above theories, given an uncountable cardinal
	$\kappa$ (with ${\rm cof}(\kappa)\geq \aleph_1$  in the last example) and a $\kappa$-prime model $\UU$ over $\cl(\paraset)$, where $\paraset$ is a subset of   size strictly less than $\kappa$,  the automorphism group 
	$\Aut(\UU/\cl(\paraset))$ is simple. 
\end{thB}

\medskip
\noindent {\bf Acknowledgements:} The authors would like to thank the anonymous referee for the comments and suggestions which have considerably improved our previous presentation.

\section{Preliminaries}\label{S:Prel}

Consider a first-order complete theory $T$ in a language 
$\LL$. We will 
work inside a
$\kappa$-saturated model $\UU$ of $T$ with $\kappa\ge |\LL|^+$, so all subsets we consider will have size 
strictly less than $\kappa$, unless explicitly stated. Furthermore, we impose that the 
$\LL$-theory $T$ is simple. 

In this section, we will list and discuss a series of 
properties, which will allow us to study the simplicity of the 
automorphism group.  

\smallskip
Recall that a type $p$ over $\emptyset$ is \emph{stationary} if for 
every subset $A$ of $\UU$, any two realizations of $p$  which are 
both independent  from $A$ have the same type over $A$. Here, independence means non-forking independence in the sense of the simple theory $T$. 

\begin{hyp}\label{H:single_gen}
	There exists a stationary (non-algebraic) type $p_0$ over 
	$\emptyset$ such that every element of $\UU$ is algebraic over 
	finitely many realizations of $p_0$.
\end{hyp}

\begin{remark}\label{R:exemples_base}
	We will mostly consider two following classes of theories satisfying 
	Property 
	\ref{H:single_gen}: 
	\begin{enumerate}[(a)]
		\item The theory $T$ is almost strongly minimal, that is, there is a 
		strongly minimal set $X$ defined over 
		$\emptyset$ such that (every point in) our model $\UU$ is algebraic over $X(\UU)$. The 
		unique non-algebraic type containing the definable set $X$ is our type $p_0$.
		\item The universe $\UU$ of $T$ has the structure of a definable 
		group $G$ without parameters. We will
		restrict our attention to groups with a unique generic type $p_0$,
		which is moreover stationary. In particular every element of
		$G$ is the product of two realisations of $p_0$.
	\end{enumerate}
\end{remark}

\noindent{\bf Convention}. From now on,
we will refer to the above type $p_0$ as \emph{the  generic type}. Given a subset $A$ of $\UU$, 
we will say that the element $b$ of $\UU$ \emph{is generic over $A$} if $b$ realises the
unique non-forking extension of $p_0$ to $A$.

\begin{definition} \label{D:orth}
	Two types $p$ in $S(A_1)$ and $q$ in $S(A_2)$ are \emph{orthogonal} (denoted $p\perp q$)
	if whenever a set $C$ contains $A_1\cup A_2$, any realisation $b_1$ of a non-forking extension of $p$ to $C$ is independent over $C$ from every realisation $b_2$
	of a non-forking extension of $q$ to $C$, that is, we have that $b_1\ind_C b_2$ whenever $b_i\ind_{A_i}C$ for $i=1, 2$. 
\end{definition}

In \cite[Section 3.5]{fWbook},  previous work of \cite{eH85} on certain closure operators extending the (model-theoretic) algebraic closure was adapted in order to produce closure operators arising from an arbitrary collection of types.   In our context,  the
closure is to be taken inside
the ambient model $\UU$ with respect to the generic type $p_0$  (so the closure does depend on the model). By Property 
\ref{H:single_gen}, the type of every tuple in $\UU$ is 
analysable with respect to the generic type, so that in our concrete case, the 
closure with respect to the  generic type $p_0$ 
over $\emptyset$  can be described as follows below. 

\begin{definition}\label{D:closure}
	Given a subset $A$ of $\UU$, an element $b$ \emph{belongs to the closure}
	$\cl(A)$ \emph{of} $A$ if every extension of
	$tp(b/A)$ is orthogonal to $p_0$, that is, for
	every
	subset $C$ containing $A$ and every generic
	element $g$ over $C$, we have that $g$ remains generic over
	$C\cup\{b\}$.  
	
	We say that the subset $X$ is \emph{cl-closed} if $\cl(X)=X$. The set $X$ is \emph{cl-generated by $A$} if $X=\cl(A)$.
\end{definition}

Note that  every element which is algebraic over $A$ (in the 
model-theoretic
sense)  belongs to $\cl(A)$.  In contrast to the algebraic
closure,  in most examples of fields with operators we are
interested in, the closure of a set will be rather large (e.g., of
cardinality $\geq\kappa$ in a $\kappa$-saturated model), and may
increase as the model changes. 
For example, in the case of 
differentially closed fields of characteristic $0$,  every constant element 
(that is, whose derivative is $0$) belongs to $\cl(\emptyset)$.

This notion of closure  satisfies some
important properties if every element is analysable with respect to $p_0$, which justifies the following assumption. 

\medskip 
{\bf Henceforth, we will assume that the simple theory $T$ satisfies 
	Property \ref{H:single_gen} with respect to the fixed generic type 
	$p_0$.}

\begin{remark}\label{R:closure}\textup{(}\cf\ \cite[Lemmata
	3.5.3 and 3.5.5]{fWbook}\textup{)}~
	\begin{enumerate}[(a)]
		\item (Usual properties of a closure operator) If $A\subset B$, then $\cl(A)\subset \cl(B)$ and
		$\cl(\cl(A))=\cl(A)$. 
		\item If $g$ is generic over $A$, then it remains so over $\cl(A)$. In particular, no generic element over $A$ lies in $\cl(A)$. 
		\item Given  two subsets $B$ and $C$ of $\UU$ with a common subset
		$A$,
		\[ B\ind_A C \quad \Longrightarrow \quad \cl(B) \ind_{\cl(A)}
		\cl(C).\]
	\end{enumerate}
\end{remark}
It follows from Local Character and the above Remark that the closure of a subset $A$ of cardinality possibly larger than $|\LL|$ is the union of all $\cl(A_0)$, where $A_0$ runs among the subsets of $A$ of cardinality at most $|\LL|$.

\medskip The stationarity of the type $p_0$ will be crucial  in many of our proofs. In a stable theory with weak elimination
of 
imaginaries, types over algebraically closed sets are always
stationary. As shown
in \cite[Theorem 5.3 and Corollary B.11]{CH04} for difference  closed
fields (which are unstable),  types of the form $\tp(a/B)$ with $B=\cl(B)\cap
\acl(B,a)$ are stationary , so this motivates the following property.
\begin{hyp}\label{H:closed_stat}
	Types over  relatively
	$\cl$-closed subsets are stationary: Given subsets $A_1$, $A_2$ and $B$
	with a common algebraically closed
	subset $C$ such that \[ \acl(A_1)\cap \cl(C)= C= \acl(A_2)\cap
	\cl(C),\] if $A_1\equiv_C A_2$ and \[ A_1\ind_C B \text{ and }
	A_2\ind_C B,\] then $A_1\equiv_B A_2$.
\end{hyp}

\medskip
The following property will be fundamental in order to extend partial
automorphisms to arbitrary subsets in terms of a small chain of extensions
obtained by successively adding generic elements.

\begin{hyp}\label{H:addgen}
	For every $\kappa$-saturated model $\M$ and every subset $\paraset$ of $\M$ of cardinality strictly less than $\kappa$, we have that $\M=\cl(\paraset \cup A)$, where $A$ enumerates a sequence  of independent realizations over $\paraset$ of the generic type. 
\end{hyp}

\begin{remark}\label{R:addgen_weaker_version}
	\begin{enumerate}[(a)]
		\item It follows from Remark \ref{R:closure} (b) that the above sequence $A$ in Property \ref{H:addgen} is maximal with respect to being independent realizations of the generic type over $\paraset$. Moreover, notice that $|A|\ge \kappa$ by $\kappa$-saturation of $\M$. 
		\item\label{I:addgen_old} If Property \ref{H:addgen} holds, then for every 
		subset $\paraset$ of the ambient model $\UU$ each element $b$ of $\UU$ 
		is contained in $\cl(\paraset\cup A_0)$, where $A_0$ enumerates a sequence of length at most $|\LL|$ of independent realizations of the generic  type over $\paraset$. Indeed, we know that $\UU=\cl(\paraset\cup A)$, where $A$ enumerates a sequence of independent realizations over $\paraset$ of the generic type. Applying  Local Character to $\tp(b/\paraset\cup A)$, we deduce from Remark \ref{R:closure} (c) that $b$ belongs to $\cl(\paraset\cup A_0)$ for some $A_0\subset A$ of cardinality at most $|\LL|$, as desired. 
	\end{enumerate}
\end{remark}

\begin{definition}\label{D:atomic}
	Given a finite tuple $c$ and a subset $K$ of  $\UU$ of cardinality possibly larger than $\kappa$, the type 
	$\tp(c/K)$ is $\kappa$-\emph{isolated} if there is some subset 
	$E$ of $K$
	of cardinality strictly less than $\kappa$ such that 
	$\tp(c/E)\vdash \tp(c/K)$.
	
	The  model $\UU$ is $\kappa$-\emph{atomic} over $K$
	if for every finite tuple $c$ of $\UU$, the type $\tp(c/K)$ is  \emph{$\kappa$-isolated}.
\end{definition}

\begin{remark}\label{R:Konnerth}
	If $\UU$ is $\kappa$-atomic over $K$, then for every $c$ in $\UU$ and every partial elementary map $f: K\to \UU$,  there is an
	extension of $f$ to a partial elementary map defined on
	$K\cup\{c\}$. Indeed, by $\kappa$-atomicity, we have that $\tp(c/E)\vdash \tp(c/K)$ for some subset $E$ of $K$ of cardinality strictly less than $\kappa$. In particular, the same holds for the images of these types under the map $f$. By $\kappa$-saturation of $\UU$, we can find a realization $d$ of $f(\tp(c/E))$, so we can extend $f$ to $K\cup\{c\}$ mapping $c$ to $d$. 
\end{remark}

The following property was 
shown by Konnerth for the theory DCF$_m$ \cite[Lemma 2.3]{rK02}. It will be  an important
tool towards extending partial isomorphisms.

\begin{definition}\label{D:Konnerth} The theory $T$ satisfies Property (WH) (for \emph{weak homogeneity}) if every $\kappa$-saturated model $\UU$ of $T$ is $\kappa$-atomic over subsets of the form 
	\[K=\acl(\cl(A_1)\cup\cdots\cup\cl(A_n)\cup B),\] such that all $A_i$'s and $B$ have cardinality strictly less than $\kappa$.

	In particular, given some $c$ in $\UU$, every partial elementary map $f: K\to \UU$ extends to a partial elementary map defined on $K\cup\{c\}$.      
\end{definition}

\section{Tame models}\label{S:tame}
{\bf In this section, we assume that   $\kappa\ge |\LL|^+$ 
	and that the simple $\LL$-theory $T$ satisfies Properties
	\ref{H:single_gen}-\ref{H:addgen}. }

In order to prove Theorem A of the introduction (which corresponds to Theorem \ref{T:main}), we will need to restrict our focus to a particular class of $\kappa$-saturated models, which we
call 
\emph{$\kappa$-tame}.
\begin{definition}\label{D:tame_model}
	A $\kappa$-saturated model $\UU$ of the theory $T$ is
	\emph{$\kappa$-tame} over the subset  $\paraset$ of size strictly less than $\kappa$ if it satisfies the following conditions:
	\begin{enumerate}[(a)]
		\item\label{I:tame_dim} There is a subset $A$ of $\UU$ of size $\kappa$ enumerating an independent sequence over $\paraset$ of
		realizations of the generic type $p_0$ (or rather of the non-forking extension of $p_0$ to $\paraset$) with $\UU=\cl(\paraset \cup A)$.
		
		\item\label{I:tame_extension} (Strong
		homogeneity) Every partial elementary map $f: K\to \UU$ fixing $\cl(\paraset)$ pointwise with
		\[ K=\acl(\cl(\paraset)\cup\cl(A_1)\cup\cdots \cup\cl(A_n)\cup B),\] and \[f(K)=\acl(\cl(\paraset)\cup\cl(f(A_1))\cup \cdots\cup \cl(f(A_n))\cup f(B)),\]                   
		\noindent   such that all
		$A_i$'s and $B$ have cardinality strictly less than
		$\kappa$, 
		extends to a global automorphism of $\UU$ fixing $\cl(\paraset)$ pointwise.
	\end{enumerate}
\end{definition}

\begin{remark}\label{R:tame_equiv}
	\begin{enumerate}[(i)] 
		\item It follows directly from condition
		(\ref{I:tame_dim}) that the model $\UU$ does not
		contain a sequence of
		length $\kappa^+$ of independent realizations over $\paraset$ of the generic type.
		\item\label{I:tame_equivalence_set} We may replace condition
		(\ref{I:tame_dim}) in Definition \ref{D:tame_model} by
		$\UU=\cl(\paraset\cup B)$ for some subset $B$ of size  $\kappa$ (not
		necessarily consisting of realizations of $p_0$).
		Indeed, one applies Remark \ref{R:addgen_weaker_version} (\ref{I:addgen_old}) to some
		enumeration of $B$ of order type $\kappa$ to obtain an
		independent set $C$ of cardinality
		at most $\kappa$ of realisations of the generic type over $\paraset$,  such that $B\subset\cl(\paraset\cup C)$, so
		$\cl(\paraset\cup C)=\UU$. Note that the set $C$ has cardinality $\kappa$, by Remark \ref{R:closure} (b) and 
		$\kappa$-saturation of $\UU$. 
	\end{enumerate}
\end{remark}

Given a subset $K=\acl(\cl(\paraset)\cup\cl(A_1)\cup\cdots \cup\cl(A_n)\cup B)$ as in Definition \ref{D:tame_model}(b) and
a sequence $(c_i)_{i<\lambda}$ with $\lambda<\kappa$, the set $\acl(K,
\{c_i\}_{i<\lambda})$ also satisfies the hypotheses of Definition
\ref{D:tame_model}(b). Thus, the following remark follows immediately  from the stationarity of the generic type in     Property \ref{H:single_gen}:
\begin{remark}\label{R:tame_ext_dim}
	Consider a $\kappa$-tame model $\UU$ over $\paraset$ of the theory $T$ and a partial
	elementary map $f: K\to f(K)\subset \UU$ as  in Definition \ref{D:tame_model}(b). Given  $\lambda<\kappa$ as well as two sequences
	$(c_i)_{i<\lambda}$ and $(d_i)_{i<\lambda}$ of realizations of the
	unique generic type which are respectively independent over $K$ and
	over $f(K)$,  there is an extension of $f$ to a
	global automorphism  
	mapping the sequence $(c_i)_{i<\lambda}$ to the sequence
	$(d_i)_{i<\lambda}$. In particular, there is a global automorphism of $\UU$ mapping  $\cl(K,
	\{c_i\}_{i<\lambda})$ onto $\cl(f(K), \{d_i\}_{i<\lambda})$ and fixing $\cl(\paraset)$ pointwise.

\end{remark}
\begin{lemma}\label{L:tame_ext_ind}
	Let $\UU$ be a $\kappa$-tame model over $\paraset$ of the theory $T$ and consider three 
	$\cl$-closed subsets $X$, $Y_1$ and $Y_2$ of $\UU$ with $\paraset \subseteq X \subseteq
	Y_1 \cap Y_2$, each $\cl$-generated over $\paraset$ by subsets of size strictly less
	than $\kappa$. Given elementary automorphisms $g_i$ 
	of $Y_i$ for $1\le i\le 2$ 
	fixing $\cl(\paraset)$ pointwise which agree on $X$ and are such that $g_1(X)=g_2(X)= X$, if 
	$Y_1 \ind_X Y_2$,
	then there is a global automorphism of $\UU$ which extends both $g_1$
	and $g_2$ (and hence fixes $\cl(\paraset)$ pointwise). 
\end{lemma}
\begin{proof} Set $Y_1 = \cl(\paraset\cup A_1)$ for some subset $A_1$ of cardinality
	stricly less than $\kappa$.  Note that $Y_1 = \cl(\paraset\cup g_1(A_1))$, since $g_1$ is an
	elementary map of $Y_1$ onto itself.
	By $\kappa$-tameness,
	there exists a global automorphism $\widehat{g_2}$ of $\UU$ extending
	$g_2$ fixing $\cl(\paraset)$ pointwise. After fixing an enumeration of $Y_1$, the subset $Y'_1 =
	\widehat{g_2}\inv(g_1(Y_1))$ has the same type as $Y_1$ over $X$, since
	$g_1$ and $g_2$ (and thus $\widehat{g_2}$) agree on $X$. Note that
	$Y'_1 \ind_X Y_2$. By Property \ref{H:closed_stat}, we have
	$Y'_1 
	\equiv_{Y_2}  Y_1$, since $X$ is $\cl$-closed. Thus, there is an
	elementary map~$h$ on $\acl(Y_1 \cup Y_2)$ fixing pointwise
	$Y_2$  whose restriction  to $Y_1$ coincides with
	$\widehat{g_2}\inv\circ  g_1$. Note that $Y'_1 = \cl(\paraset\cup h(A_1))$, so we obtain
	by tameness  a global automorphism $\widehat{h}$ of $\UU$ extending
	$h$. By construction, the automorphism  $\widehat{g_2} \circ
	\widehat{h}$ extends both $g_1$ and $g_2$.
\end{proof}

Under the additional assumption that the simple theory  $T$ satisfies
Property (WH)(Definition \ref{D:Konnerth}), there are  natural
examples of $\kappa$-tame models:  saturated models of
cardinality $\kappa$ and $\kappa$-prime models of $T$ are $\kappa$-tame.  

\begin{lemma} \label{L:sat_tame}
	Assume the simple theory $T$ satisfies Properties
	\ref{H:single_gen}-\ref{H:addgen} and (WH). If 
	$\UU$ is a saturated model of $T$ 
	of cardinality $\kappa$, then  $\UU$ is
	$\kappa$-tame over every subset $\paraset$ of size strictly less than $\kappa$.
\end{lemma}
\begin{proof}
	Condition (a) of Definition \ref{D:tame_model}  holds
	trivially, since the model $\UU$ has cardinality $\kappa$,
	by Remark
	\ref{R:tame_equiv}(\ref{I:tame_equivalence_set}).  Condition (b) holds clearly by a standard Back-\&-Forth argument, using Property (WH) (and Remark \ref{R:Konnerth}). 
\end{proof}

Whenever $T$
is stable,  saturated models of cardinality $\kappa$ 
exist if and only if $T$ is stable at the cardinal $\kappa$ (\cite{vH75,
	sSbook}). For a general theory $T$, if $\kappa=\lambda^+$ for
some cardinal $\lambda\ge |\LL|$ with $\lambda^+=2^\lambda$, or
if $\kappa$ is regular and strongly inacessible, then there are
saturated models of cardinality $\kappa$. However, these
set-theoretic  assumptions go beyond ZFC.

Shelah introduced the notion of $\kappa$-prime models \cite[Chapter
IV]{sSbook} (see also \cite[Chapitre VI]{dLBook}) for an arbitrary
stable theory, generalizing Morley's notion of prime models for an
$\omega$-stable theory.

\begin{definition}\label{D:prime}
	A model $\UU$ of $T$ is
	$\kappa$-\emph{prime} over  $A\subseteq \UU$ if it
	is $\kappa$-saturated and 
	elementarily embeds over $A$ into every 
	$\kappa$-saturated model of $T$ 
	containing $A$.
\end{definition}

\begin{remark}\label{R:Exist} 
	Shelah \cite[Theorem  IV.4.14 and IV.4.18]{sSbook} showed the existence and uniqueness (up to isomorphism) of
	$\kappa$-prime models over arbitrary subsets $A$ (even if the size
	of $A$ is greater than  $\kappa$) whenever the 
	theory is stable, as long as the cofinality of $\kappa$ is
	at least $|\LL|^+$. Indeed, Shelah shows that a $\kappa$-saturated model of the stable theory is $\kappa$-prime over $A$ if and only if it is 
	$\kappa$-atomic over $A$ and contains no (non-constant)
	$A$-indiscernible sequence of length $\kappa^+$.
	
	In the particular case that $T$ is superstable, then the same holds for all  $\kappa\geq|\LL|^+$, regardless of the cofinality.
	
	Whilst every  completion $T$ of the theory
	ACFA$_0$ is simple yet unstable,  the second author obtained in \cite{zC23} an analogous result to the existence and uniqueness of $\kappa$-prime models  for ACFA$_0$ as well as their characterisation, with no restriction on
	$\kappa\geq\aleph_1$, but assuming that the base set $A$ is algebraically closed 
	and that the fixed field of $A$ is
	a $\kappa$-saturated pseudo-finite field.      
\end{remark}
Motivated by the above description of $\kappa$-prime models, we show the following result:  
\begin{prop}\label{P:prime_tame}
	Assume that the  simple theory $T$ satisfies Properties
	\ref{H:single_gen}-\ref{H:addgen} and (WH). We suppose furthermore the following condition for every $\kappa$-saturated model $\M$:
	
	For every subset of $\M$ of the form $K=\acl(\cl(A_1)\cup\cdots \cup\cl(A_n)\cup B),$                    
	\noindent   such that all
	$A_i$'s and $B$ have cardinality strictly less than
	$\kappa$, there exists a unique $\kappa$-prime model over
	$K$, which is characterized by being $\kappa$-atomic over $K$ and
	containing no (non-constant) $K$-indiscernible sequence of length
	$\kappa^+$.

	Then for every $\kappa$-saturated model $\M$ of $T$ and
	${\paraset}\subseteq \M$ of cardinality strictly less than $\kappa$,  the $\kappa$-prime
	model $\UU$ over $\cl_\M({\paraset})$ is $\kappa$-tame over $\paraset$ and
	$\cl_\UU({\paraset})=\cl_\M({\paraset})$, where $\cl_\UU$ and $\cl_\M$ denote the closures taken within the models  $\UU$ and $\M$ respectively.
\end{prop}

\begin{proof}
	By assumption, the $\kappa$-prime model $\UU$ over $\cl_{\M}(\paraset)$ exists. Since $\UU$ embeds into $\mathcal M$ over $\cl_{\M}(\paraset)$, it
	follows immediately that $\cl_\UU(\paraset)=\cl_\M(\paraset)$. From now on, we will just write $\cl(Z)$ without referring to a particular model. Let us now show that each of the conditions in Definition \ref{D:tame_model} of $\kappa$-tameness holds for the $\kappa$-prime model $\UU$ over $\cl(\paraset)$. For the first condition, Property \ref{H:addgen},  yields that  $\UU=\cl(\paraset \cup A)$, where $A$ is a set enumerating a maximal sequence  of independent realizations over $\paraset$ of the generic type (Note that $|A|\ge \kappa$ by saturation). Since the sequence determined by $A$ is $\paraset$-indiscernible by the stationarity of the generic type $p_0$ in Property \ref{H:single_gen}, we conclude from our assumption in the statement that $A$ has cardinality exactly $\kappa$, as desired.
	
	We now show the strong extension property for elementary maps $f: K\to
	\UU$ such that there exists subsets $A_i$'s and $B$ of $\UU$
	of 
	cardinality strictly less than $\kappa$ with \begin{itemize} \item $K=\acl(\cl(\paraset)\cup\cl(A_1)\cup\cdots
		\cup\cl(A_n)\cup B)$;  \item $f(K)=
		\acl(\cl(\paraset)\cup \cl(f(A_1))\cup\cdots\cup\cl(f(A_n))\cup
		f(B))$; 
		\item the map $f$ is the identity on $\cl(\paraset)$.
	\end{itemize}
	Now,  Property 
	(WH) gives that  the $\kappa$-saturated model $\UU$ is $\kappa$-atomic over $K$ and over $f(K)$ as well. Our above characterisation of $\kappa$-prime models implies that $\UU$ contains no non-constant indiscernible sequence of
	length 
	$\kappa^+$ over $\cl(\paraset)$, so $\UU$ does not have such indiscernible sequences over  $K$ nor over $f(K)$. Hence, by the characterization and the uniqueness of
	$\kappa$-prime models in $(ii)$, the partial isomorphism $f:K\to f(K)$
	extends to
	an automorphism of $\UU$.
\end{proof}

\section{Automorphisms of $\kappa$-tame models}\label{S:Lascar}

{\bf In this section we fix a cardinal $\kappa \ge |\LL|^+$ and a simple $\LL$-theory
	$T$, which 
	satisfies Properties  \ref{H:single_gen}-\ref{H:addgen}. We assume furthermore that $T$ has a
	$\kappa$-tame model $\UU$ over a subset ${\paraset}$ of
	cardinality strictly less than $\kappa$.}

\begin{definition}\label{D:familyS}
	We will denote by $\clS$ the collection of subsets of $\UU$ of the form
	$\cl(\paraset \cup A)$, where $A$ enumerates a sequence  of length strictly less than $\kappa$ of independent realizations
	of the generic type over $\paraset$. 
\end{definition}
Note that             each member of $\clS$ is an algebraically closed substructure of $\UU$. Moreover, if $B$ enumerates a sequence of length strictly less than $\kappa$ of independent realizations
of the generic type over $X$ (by Remark \ref{R:closure} it suffices that the  sequence is independent over $D$ with $X=\cl(D)$), then $Y=\cl(X\cup B)$ lies again in $\clS$.

\begin{definition}\label{D:acceptable}
	An extension $X\subseteq Y$ with $X$ and $Y$ in $\clS$ is
	\emph{acceptable} if $Y=\cl(X \cup A)$, where $A$ enumerates a sequence of length  $|\LL|$ of independent realizations of the generic
	type over $X$.
\end{definition}
\begin{remark}\label{R:acceptable}
	For $\alpha<\kappa$, every increasing union $(X_\beta)_{\beta<\alpha}$ in $\clS$ with $X_{\beta+1}$ acceptable over $X_{\beta}$ and $X_\gamma =\bigcup_{\beta<\gamma} X_\beta$ for $\gamma<\alpha$ a limit ordinal belongs again to $\clS$.
\end{remark}

Remark \ref{R:tame_ext_dim} yields immediately the following result:
\begin{lemma}\label{L:iso_accept}
	Given an elementary automorphism $f$ of $X$ in $\clS$ and two
	acceptable extensions $Y_1$ and $Y_2$ of $X$, there is an extension of
	$f$ to an elementary partial map sending $Y_1$ onto $Y_2$. In particular, any two
	acceptable extensions $Y_1$ and $Y_2$ of $X$ are conjugate by an
	automorphism of $\UU$ fixing $X$ pointwise. \qed
\end{lemma}

\begin{remark}\label{R:iteration}
	Given $X$ in $\clS$, an element $b$ in $\UU$ as well as a countable
	collection $\mathcal F$ of global automorphisms of $\UU$ which 
	leaves $X$ setwise invariant,  there exists an acceptable extension $Y$
	of $X$ in $\clS$ which contains $b$ and is setwise stable
	under each automorphism of $\mathcal F$.
\end{remark}
\begin{proof}
	Without loss of generality, we may assume that $\mathcal F$ is a
	group (under composition).
	By induction, successively applying Remark \ref{R:addgen_weaker_version} (\ref{I:addgen_old}), we
	construct a increasing sequence $(A_n)_{n\in \N}$ of sets such that:
	\begin{itemize}
		\item the set $A_0$ $\cl$-generates $X$;
		\item the element $b$ belongs to $\cl(A_1)$;
		\item for every $n\ge 1$ in $\N$, every $\tau$ in $\mathcal F$ and every
		$a$ in $A_n \setminus A_{n-1}$, the element $\tau(a)$ belongs to $\cl(A_{n+1})$.
		\item the elements in $A_{n+1}\setminus A_n$ enumerate an 
		independent sequence of realizations of length  $|\LL|$ of the generic type over $A_n$ 
		(and thus over $\cl(A_n)\supset X$);  
	\end{itemize}
	The set $Y=\cl\left(\bigcup_{n\in \N} A_n\right)$ has all the
	desired properties.
\end{proof}

We recall now a modified version of when an automorphism \emph{moves maximally}, according to the
terminology of
Tent and Ziegler (\cite[Definition 2.5]{TZ13}) restricting our attention to the unique generic type $p_0$. 

\begin{definition}\label{D:moves_max}
	An automorphism $\tau$ of a $\kappa$-saturated model $\M$ of $T$ \emph{moves $p_0$ maximally} if for every subset $A$ of $\M$ of cardinality strictly less than $\kappa$ which is stable under the action of $\tau$, there exists a generic element $b$ over $A$ such that $\tau(b)$ is generic over $A\cup\{b\}$, or equivalently, such that  $\tau(b)\ind_A b$. 
\end{definition}
\begin{remark}\label{R:moves_max}
	If $\tau$ moves maximally, then for every subset $A$ of cardinality strictly less than $\kappa$, which need not be stable under the action of $\tau$, there is some generic element $b$ over $\bigcup_{k\in \Z} \tau^k(A)$ with $\tau(b)\ind_A b$. In particular, both $b$ and $\tau(b)$ are generic over $A$.   
\end{remark}
\begin{lemma}\label{L:imageInd} Let $\tau$ be an automorphism of
	$\UU$ moving $p_0$ maximally. If $X$ in
	$\clS$ is stable setwise under the action of $\tau$,  then there is an
	acceptable   extension $Y$ of $X$ such  that $Y$ and $\tau(Y)$ are
	independent over $X$.
\end{lemma}
Notice that $\tau(Y)$ is again an acceptable extension of $X$.
\begin{proof} Set $X=\cl(D)$ for some $D\supset \paraset$ of cardinality strictly less than $\kappa$.   
	By Remark \ref{R:closure} (c), it suffices to inductively
	build an independent sequence $(a_\alpha)_{\alpha<|\LL|}$ of realizations of the
	generic type over $D$ such that for every $\alpha<|\LL|$ we have 
	\[ a_\alpha\ind_{D} (a_\beta, \tau(a_\beta))_{\beta<\alpha} \quad
	\text{ and } \quad \tau(a_\alpha)\ind_{D} (a_\beta, \tau(a_\beta))_{\beta<\alpha} \cup \{a_\alpha\}.\] Indeed, the set $Y=\cl(X \cup \{a_\alpha\}_{\alpha<|\LL|})$ belongs to $\clS$ and is an acceptable extension of $X$ with the desired properties.
	
	Suppose the sequence has been constructed for all $\beta<\alpha$. Now, the automorphism $\tau$ moves $p_0$ maximally over 
	$D'= D\cup (a_\beta, \tau(a_\beta))_{\beta<\alpha}$, so by Remark \ref{R:moves_max} we find an element $a_\alpha$ in
	$\UU$ generic over $D'$ such that  \[ \tau(a_\alpha)\ind_{D'} a_\alpha.\] In particular, the element $a_\alpha$ is generic over $D\cup\{a_\beta\}_{\beta<\alpha}$, as desired.
\end{proof}

We will now reproduce verbatim  Lascar's proof of the
simplicity 
of the group of field automorphisms of $\mathbb C$ fixing $\Q^\text{alg}$
\cite{dL97}. More precisely, we will show that every element in $\Aut(\UU/\cl(\paraset))$ is a product of four conjugates of $\tau$ and $\tau\inv$, whenever $\tau$ moves $p_0$ maximally (in case such an automorphism exists).

We fix such an automorphism $\tau$ of $\UU$ fixing $\cl(\paraset)$
pointwise. Denote by \[\Psi(f,g)= \tau^{f}\circ (\tau\inv)^{f\circ
	g}=f\circ \tau \circ g\circ \tau\inv\circ g\inv\circ f\inv=[\tau,
g]^f\] for $f$ and $g$ automorphisms of $\UU$, where $\tau^f=f\circ
\tau\circ f\inv$ and $[\tau,g]=\tau g\tau^{-1}g^{-1}$.

Whenever $X$ in $\clS$ is stable under the action of $\tau$, given $f$ and $g$ elementary automorphisms of $X$, we denote
\[\Psi_X(f,g)= (\tau\restr{X})^{f}\circ (\tau\restr{X}\inv)^{f\circ g}
=\Psi(f,g)\restr{X}.
\] 

We will show in \ref{T:main} that every automorphism $\nu$ of $\UU$ fixing
$\cl(\paraset)$ can be written as \[\nu=\Psi(f_1,f_2) \circ
\Psi(f_3,f_4)\inv,\] for suitable automorphisms $f_1,\ldots, f_4$. This will be done by a chain of
approximations of $\nu$ to smaller substructures in the class $\clS$.

If the partial isomorphism $g$ extends the partial isomorphism $f$, we will
denote it by $f\subset g$. For the back-and-forth process to describe $\nu$ as
a product of conjugates of $\tau$ and $\tau\inv$, we will need the following
central result.

\begin{prop}\label{P:central} Consider $X$ in $\clS$ stable under the action of
	the automorphism $\tau$ which moves $p_0$ maximally and fixes 
	$\cl(\paraset)$ pointwise, two elementary automorphisms $f$ and $g$ of $X$ as well as an acceptable
	extension $Y\supseteq X$ in $\clS$ equipped with an elementary automorphism $h$
	of $Y$ extending
	$\Psi_X(f,g)$.  There exist two automorphisms $f'$ and $g'$ of $\UU$
	extending $f$ and $g$, respectively, such that
	\[ \Psi(f',g') \supset h \supset \Psi_X(f,g),\] so $\Psi(f',g')\restr Y=h$. 
\end{prop}
\begin{proof}
	By Lemma \ref{L:imageInd}, there exists an acceptable extension $Y_1$ of
	$X$ such that $Y_1$ and $Y_2=\tau(Y_1)$ are independent over $X$. By Lemma
	\ref{L:iso_accept}, there is an automorphism $f'$ of $\UU$ extending $f$
	such that $f'$ maps $Y_2$ onto $Y$.
	Consider now the elementary automorphism
	\[h_2 = f'^{-1}\circ h \circ f'\]
	of $Y_2$, which restricted to $X$ equals \[ f\inv \circ h\restr X \circ f=
	f\inv \circ\Psi_X(f,g) \circ f =\tau\restr X \circ (\tau\restr X \inv)^g.\]
	Again by  Lemma \ref{L:iso_accept}, choose an elementary automorphism $g_2$
	extending $g$ which maps $Y_2$ onto itself. The elementary automorphism \[
	g_1= (h_2\circ g_2)^{\tau\inv}=\tau\inv\circ h_2\circ g_2\circ \tau\] of
	$Y_1$ restricted to $X$ extends \[ \tau\restr X \inv\circ h_2\restr X\circ
	g \circ \tau\restr X= (\tau\restr X \inv)^g \circ g\circ \tau\restr X =
	g.\]
	
	We include a diagram to facilitate the
	arrow-chasing:
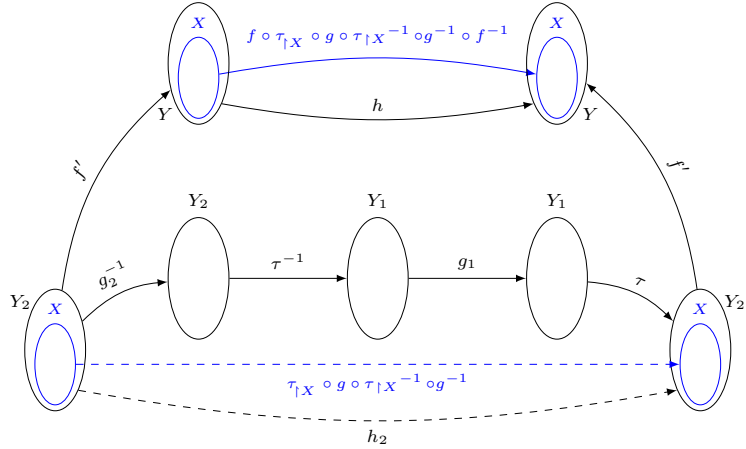
\begin{figure}[H]
		\centering
		\resizebox{0.75\textwidth}{!}{%
			\begin{tikzpicture}
				\draw(2,7) node (a0) [blue,ellipse, draw, inner xsep=  0.2cm,
				inner ysep= 0.4cm, label={[blue]{\tiny $X$}}]{};
				\draw(7,7) node (b0) [blue,ellipse, draw, inner xsep=  0.2cm,
				inner ysep= 0.4cm, label={[blue]{\tiny $X$}}]{};
				\draw[->,>=latex,blue] (a0) to[bend left=10]
				node[above]
				{\tiny $f\circ
					\tau_{\restr X} \circ g \circ\tau \restr X ^{-1} \circ
					g^{-1} \circ
					f^{-1}$} (b0);
				
				\draw(2,7.2) node (a1) [black,ellipse, draw, inner xsep=
				0.3cm, inner ysep= 0.6cm, label={[label
					distance=-0.15cm,black]245:{\scriptsize $Y$}}]{};
				\draw(7,7.2) node (b1) [black,ellipse, draw, inner xsep=
				0.3cm, inner ysep= 0.6cm, label={[label
					distance=-0.15cm,black]295:{\scriptsize $Y$}}]{};
				\draw[->,>=latex,black] (a1.300) to[bend right=10]
				node[sloped,above] {\scriptsize $h$} (b1.240);
				
				\draw(0,3.2) node (c1) [black,ellipse, draw, inner xsep=
				0.3cm, inner ysep= 0.6cm, label={[label
					distance=-0.2cm,black]115:{\scriptsize $Y_2$}}]{};
				\draw[->,>=latex,black] (c1) to[bend left=20]
				node[sloped,above] {\scriptsize $f'$} (a1);
				
				\draw(9,3.2) node (d1) [black,ellipse, draw, inner xsep=
				0.3cm, inner ysep= 0.6cm, label={[label
					distance=-0.2cm,black]65:{\scriptsize $Y_2$}}]{};
				\draw[->,>=latex,black] (d1) to[bend right=20]
				node[sloped,above] {\scriptsize $f'$} (b1.-34);
				
				\draw[->,>=latex,black,dashed] (c1.300) to[bend right=10]
				node[sloped,below] {\scriptsize $h_2$} (d1.240);
				\draw(0,3) node (c0) [blue,ellipse, draw, inner xsep=  0.2cm,
				inner ysep= 0.4cm, label={[blue]{\tiny $X$}}]{};
				\draw(9,3) node (d0) [blue,ellipse, draw, inner xsep=  0.2cm,
				inner ysep= 0.4cm, label={[blue]{\tiny $X$}}]{};
				\draw[->,>=latex,blue,dashed] (c0) to node[below] {\tiny
					$\tau_{\restr X} \circ g \circ\tau \restr X ^{-1} \circ
					g^{-1}$} (d0);
				
				\draw(2,4.2) node (e1) [black,ellipse, draw, inner xsep=
				0.3cm, inner ysep= 0.6cm, label={[black]{\scriptsize
						$Y_2$}}]{};
				\draw[->,>=latex,black] (c1) to[bend left=20]
				node[sloped,above] {\scriptsize $g_2^{-1}$} (e1);
				
				\draw(7,4.2) node (f1) [black,ellipse, draw, inner xsep=
				0.3cm, inner ysep= 0.6cm, label={[black]{\scriptsize
						$Y_1$}}]{};
				\draw[->,>=latex,black] (f1) to[bend left=20]
				node[sloped,above] {\scriptsize $\tau$} (d1);
				
				\draw(4.5,4.2) node (g1) [black,ellipse, draw, inner xsep=
				0.3cm, inner ysep= 0.6cm, label={[black]{\scriptsize
						$Y_1$}}]{};
				\draw[->,>=latex,black] (e1) to node[above] {\scriptsize
					$\tau^{-1}$} (g1);
				\draw[->,>=latex,black] (g1) to node[above] {\scriptsize
					$g_1$} (f1);
				
			\end{tikzpicture}
		}
	\end{figure}
	
	Lemma \ref{L:tame_ext_ind} yields now a common extension $g'$ to $\UU$ of
	the elementary automorphisms $g_1$ of $Y_1$ and $g_2$ of $Y_2$. We need
	only check now that the restriction $\Psi(f', g')\restr Y$ of the global automorphism $\Psi(f', g')$ equals $h$, or
	equivalently, that $f'^{-1}\circ \Psi(f', g') \circ f'$ extends $h_2$. By
	the definition of $\Psi(f',g')$, we have that the global automorphism \[  f'^{-1}\circ \Psi(f', g')
	\circ f' = \tau \circ (\tau\inv)^{g'}\]  equals $h_2 = \tau \circ
	g_1 \circ \tau\inv \circ g_2\inv$ on $Y_2$, as desired.
\end{proof}

The previous Proposition \ref{P:central} contains all the ingredients to tackle
the successor stage of the back-and-forth construction required in the proof of
Theorem \ref{T:main}

\begin{prop}\label{P:succ}
	Let $\nu$ be an automorphism of $\UU$ fixing pointwise
	$\cl(\paraset)$. Consider an automorphism $\tau$ which moves $p_0$ maximally and fixes  
	$\cl(\paraset)$ pointwise as well as a set $X$ in $\clS$ which is stable
	under the action of both $\tau $ and $\nu$ such that \[\nu\restr X =
	\Psi_X(f_1,f_2)\circ \Psi_X(f_3, f_4)\inv\] for some elementary
	automorphisms $f_i$ of $X$, with $i=1,\ldots, 4$.
	
	For every element $a$ in $\UU$, there are:
	\begin{itemize}
		\item an acceptable extension $Y\supseteq X$ containing the element $a$  which is
		stable under both $\tau$ and $\nu$;
		\item elementary extensions $f'_i$ of $f_i$, for $i=1,\ldots, 4$, to
		$Y$, \end{itemize} such that
	\[\nu\restr{Y} = \Psi_Y(f'_1,f'_2)\circ \Psi_Y(f'_3, f'_4)\inv.\]
	
\end{prop}
\begin{proof}
	By Remark \ref{R:iteration}, there is an acceptable extension $Y_1$ of $X$
	containing $a$ invariant under the action of both $\tau$ and $\nu$. By
	Lemma \ref{L:iso_accept}, choose two elementary automorphisms $f_{3,1}$ and
	$f_{4,1}$ of $Y_1$ extending respectively $f_3$ and $f_4$. Set
	$h_1=\nu\restr{Y_1}\circ \Psi_{Y_1}(f_{3,1}, f_{4,1})$ and notice that \[
	h_1\supset \Psi_X(f_1,f_2).\] By Proposition \ref{P:central}, there are two
	global automorphisms  $f_{1,1}$ and $f_{2,1}$ of $\UU$, extending $f_1$ and
	$f_2$ such that $\Psi(f_{1,1}, f_{2,1})$  extends $h_1$.
	
	Similarly, we find an acceptable extension $Y_2$ of $Y_1$ stable under the
	action of $\tau$, $\nu$, $f_{1,1}$ and $f_{2,1}$. Then, we denote  by
	$f_{1,2}$ and $f_{2,2}$ the restrictions of $f_{1,1}$ and $f_{2,1}$ to
	$Y_2$.
	
	Notice that \[ \nu\restr{Y_2}\inv \circ \Psi_{Y_2}(f_{1,2}, f_{2,2})\supset
	\Psi_{Y_1}(f_{3, 1}, f_{4,1}).\]
	
	Iterating the above argument countably many times, we construct an
	increasing chain $Y_n$ of acceptable extensions (setting $Y_0=X$) and
	compatible elementary extensions $f_{1,2n}$ and $f_{2,2n}$ of $f_1=f_{1,0}$
	and $f_2=f_{2,0}$ to $Y_{2n}$ as well as $f_{3,2n+1}$ and $f_{4,2n+1}$ of
	$f_3$ and $f_4$ to $Y_{2n+1}$ such that for all $n$ in $\N$,
	
	\[f_{1,2n+2} \supset f_{1,2n}, \quad f_{2,2n+2} \supset f_{2,2n}, \quad
	f_{3,2n+3} \supset f_{3,2n+1}, \quad f_{4,2n+3} \supset f_{4,2n+1},\]
	\[ \nu\restr{Y_{2n+1}} \circ \Psi_{Y_{2n+1}}(f_{3,2n+1}, f_{4,2n+1})\supset
	\Psi_{Y_{2n}}(f_{1, 2n}, f_{2,2n})\]
	and
	\[ \nu\restr{Y_{2n+2}}\inv \circ \Psi_{Y_{2n+2}}(f_{1,2n+2},
	f_{2,2n+2})\supset \Psi_{Y_{2n+1}}(f_{3, 2n+1}, f_{4,2n+1}).\]
	\begin{figure}[H]
		\centering
		\resizebox{0.75\textwidth}{!}{%
			\begin{tikzpicture}
				
				\draw(0,0) node (a0) [ellipse, draw, inner xsep=  0.2cm, inner
				ysep= 0.4cm, label={{\tiny $X$}}]{};
				\draw(7,0) node (b0) [ellipse, draw, inner xsep=  0.2cm, inner
				ysep= 0.4cm, label={{\tiny $X$}}]{};
				\draw(3.5,-5) node (c0) [ellipse, draw, inner xsep=  0.2cm,
				inner ysep= 0.4cm, label={{\tiny $X$}}]{};
				\draw[->,>=latex] (a0.320) to node[above] {\tiny $\nu \restr
					X$} (b0.220);
				\draw[->,>=latex] (c0) to node[sloped,above] {\tiny
					$\Psi_X(f_3,f_4)$} (a0);
				\draw[->,>=latex] (c0) to node[sloped,above] {\tiny
					$\Psi_X(f_1,f_2)$} (b0);
				
				\node[blue,xshift=-0.3cm] () [label={[xshift=-0.15cm,
					yshift=-0.3cm]85:{\color{blue} }}] at (a0.85)  {\scriptsize
					$a$};
				\node[blue,xshift=-0.3cm] () [label={[xshift=-0.15cm,
					yshift=-0.3cm]85:{\color{blue} }}] at (c0.85) {\scriptsize $a$}
				;
				\node[blue,xshift=-0.3cm] () [label={[xshift=-0.15cm,
					yshift=-0.3cm]85:{\color{blue} }}] at (b0.85) {\scriptsize $a$};
				
				\draw(0,0.2) node (a1) [blue,ellipse, draw, inner xsep=  0.3cm,
				inner ysep= 0.6cm, label={[blue]{\tiny $Y_1$}}]{};
				\draw(3.5,-4.8) node (c1) [blue,ellipse, draw, inner xsep=
				0.3cm, inner ysep= 0.6cm, label={[blue]{\tiny $Y_1$}}]{};
				\draw[->,>=latex,blue] (c1.180) to[bend left=10]
				node[sloped,above] {\tiny $\Psi_{Y_1}(f_{3,1},f_{4,1})$}
				(a1.270);
				
				\draw(7,0.2) node (b1) [blue,ellipse, draw, inner xsep=  0.3cm,
				inner ysep= 0.6cm, label={[blue]{\tiny $Y_1$}}]{};
				\draw[->,>=latex,blue] (a1.330) to[bend left=10]
				node[sloped,above] {\tiny $\nu \restr{Y_1}$} (b1.210);
				\draw[->,>=latex,blue] (c1.0) to[bend right=10]
				node[sloped,above] {\tiny $h_1$} (b1.270);
				
				\draw(7,0.45) node (b2) [red,ellipse, draw, inner xsep=
				0.45cm, inner ysep= 0.9cm, label={[red]{\tiny $Y_2$}}]{};
				\draw(3.5,-4.55) node (c2) [red,ellipse, draw, inner xsep=
				0.45cm, inner ysep= 0.9cm, label={[red]{\tiny $Y_2$}}]{};
				\draw[->,>=latex,red] (c2.320) to[bend right=20]
				node[sloped,above] {\tiny $\Psi_{Y_2}(f_{1,2},f_{2,2})$}
				(b2.290);
				
				\draw(0,0.45) node (a2) [red,ellipse, draw, inner xsep=
				0.45cm, inner ysep= 0.9cm, label={[red]{\tiny $Y_2$}}]{};
				\draw[->,>=latex,red] (a2.340) to[bend left=20]
				node[sloped,above] {\tiny $\nu \restr{Y_2}$} (b2.200);
				\draw[->,>=latex,red] (c2.220) to[bend left=20]
				node[sloped,above] {\tiny $h_2$} (a2.250);
				
				\draw(0,0.65) node (a3) [purple,ellipse, draw, inner xsep=
				0.55cm, inner ysep= 1.1cm, label={[purple]{\tiny $Y_3$}}]{};
				\draw(3.5,-4.35) node (c3) [purple,ellipse, draw, inner xsep=
				0.55cm, inner ysep= 1.1cm, label={[purple]{\tiny $Y_3$}}]{};
				\draw[->,>=latex,purple] (c3.240) to[bend left=35]
				node[sloped,above] {\tiny $\Psi_{Y_3}(f_{3,3},f_{4,3})$}
				(a3.230);
				
				\draw(7,0.65) node (b3) [purple,ellipse, draw, inner xsep=
				0.55cm, inner ysep= 1.1cm, label={[purple]{\tiny $Y_3$}}]{};
				\draw[->,>=latex,purple,dashed] (a3.350) to[bend left=30]
				(b3.190);
				\draw[->,>=latex,purple,dashed] (c3.300) to[bend right=35]
				(b3.310);
				
				\node[yshift=0.75cm] () at (a3.north) {$\vdots$};
				\node[yshift=0.75cm] () at (b3.north) {$\vdots$};
				\node[yshift=0.75cm] () at (c3.north) {$\vdots$};
				\draw(0,1.2) node (a4) [ellipse, draw, inner xsep=  0.8cm,
				inner ysep= 1.6cm, label={{$Y$}}]{};
				\draw(7,1.2) node (b4) [ellipse, draw, inner xsep=  0.8cm,
				inner ysep= 1.6cm, label={{$Y$}}]{};
				\draw(3.5,-3.8) node (c4) [ellipse, draw, inner xsep=  0.8cm,
				inner ysep= 1.6cm, label={{ $Y$}}]{};
				\draw[->,>=latex] (c4.260) to[bend left=40] node[sloped,below]
				{$\Psi_{Y}(f'_3,f'_4)$} (a4.200);
				\draw[->,>=latex] (c4.280) to[bend right=40] node[sloped,below]
				{$\Psi_{Y}(f'_1,f'_2)$} (b4.340);
				\draw[->,>=latex] (a4) to[bend left=20] node[sloped,above]
				{$\nu \restr Y$} (b4);
			\end{tikzpicture}
		}
	\end{figure}
	
	By construction of the chain, the subset $Y=\bigcup_{n\in \N}
	Y_n$ lies in $\clS$ and is an acceptable extension of $X$. For $1\le i\le
	4$, denote $f'_i$
	the corresponding elementary automorphism to $Y$ given by the $f_{i,
		k}$'s.
	By construction, we have that \[\nu\restr{Y} =
	\Psi_{Y}(f'_1,f'_2)\circ \Psi_{Y}(f'_3, f'_4)\inv ,\] as desired.
\end{proof}
We have now all the ingredients to prove the simplicity, up to bounded
automorphisms, of the automorphism group of $\UU$ which fix pointwise
$\cl(\paraset)$.

\begin{theorem}\label{T:main}
	Consider a $\kappa$-tame model $\UU$ over $\paraset$ of a simple theory $T$
	satisfying Properties  \ref{H:single_gen}-\ref{H:addgen}. Assume that there is an
	automorphim $\tau$ of $\UU$ moving $p_0$ maximally and  fixing $\cl(\paraset)$ pointwise.
	Every  automorphism $\nu$ of $\UU$ fixing $\cl(\paraset)$ pointwise can be
	written as the product of four conjugates of $\tau$ and $\tau\inv$.
\end{theorem}
\begin{proof}
	By $\kappa$-tameness over $\paraset$, write $\UU=\cl( \paraset\cup \{a_\alpha\}_{\alpha<\kappa})$, where
	$(a_\alpha)_{\alpha<\kappa}$ is an independent sequence of realizations of
	the generic type over $\paraset$.  Given an automorphism $\nu$ of $\UU$ fixing
	$\cl(\paraset)$, we construct recursively a increasing chain of subsets
	$X_\alpha$, for $\alpha<\kappa$, in $\clS$ such that the extension
	$X_\alpha\subseteq X_{\alpha+1}$ is acceptable and each $X_\alpha$ is
	stable under the action of $\tau$ and $\nu$ and $\cl$-generated by a
	independent sequence of length at most $\max\big(|\LL|, |\alpha|\big)$
	of realizations of the generic type,  equipped with compatible elementary
	automorphisms $f_{i, \alpha}$ of $X_\alpha$, for $1\le i\le 4$, such that
	$a_\alpha$ lies in $X_{\alpha+1}$ and  \[\nu\restr{X_\alpha} =
	\Psi_{X_\alpha}(f_{1, \alpha},f_{2,\alpha})\circ
	\Psi_{X_\alpha}(f_{3,\alpha}, f_{4,\alpha})\inv.\]
	
	For the beginning of the recursion, set $X_0 = \cl(\paraset)$ and $f_{i,0}
	= \mathrm{Id}_{X_0}$. Assume now that $X_\beta$ has already been
	constructed for $\beta<\alpha $. If $\alpha$ is a limit ordinal, the union
	$X_\alpha = \bigcup_{\beta <\alpha} X_\beta$ belongs to $\clS$ by Remark \ref{R:acceptable} for it is $\cl$-generated
	by an independent sequence of length $\max\big(|\LL|, |\alpha|\big)$ over $\paraset$. Moreover, the automorphism  $\nu$ restricted to
	$X_\alpha$ satisfies the above identity with $f_{i,\alpha} =
	\bigcup_{\beta <\alpha} f_{i,\beta}$ for $1\le i\le 4$. If $\alpha$ is the successor of
	$\beta$, Proposition \ref{P:succ} applied to $X_\beta$ yields an acceptable
	extension $X_{\beta+1}$ of $X_\beta$ containing $a_\beta$ and elementary
	extensions $f_{i,\beta+1}$, as desired.
	
	Finally, the union $\bigcup_{\alpha<\kappa} X_\alpha$ is $\cl$-closed and contains all $a_\alpha$'s, so
	it must equal $\UU$. By construction, the automorphism $\nu$ equals a
	product of four conjugates of $\tau$ and $\tau\inv$ globally on $\UU$,
	since at every step the automorphisms are compatible.
\end{proof}

Motivated by the results in \cite{dL92,BHMP17,fW20}, we will now introduce the last property of
interest for our purposes. 

\begin{hyp}\label{H:nobounded}
	For every $\kappa$-saturated model $\M$, every non-trivial automorphism of $\M$ fixing pointwise the closure
	$\cl(\emptyset)$ moves $p_0$ maximally. 
\end{hyp}

\begin{remark}\label{R:nobounded} 
	If Property \ref{H:nobounded} holds, every non-trivial automorphism of the $\kappa$-tame model $\UU$ over $\paraset$ fixing    $\cl(\paraset)$ pointwise moves $p_0$ maximally.  
\end{remark}

If our theory $T$ satisfies Property \ref{H:nobounded}, then we get a strengthening of Theorem \ref{T:main}

\begin{cor}\label{C:main}
	For every $\kappa$-tame model $\UU$ over $\paraset$ of a simple theory $T$
	satisfying Properties  \ref{H:single_gen}-\ref{H:nobounded}, the group $\Aut(\UU/\cl(\paraset))$ of automorphisms of
	$\UU$ fixing  $\cl(\paraset)$ pointwise is simple. 
\end{cor}

In Remark \ref{R:unbounded_movemax}, we will show that, under some mild conditions, the collection of automorphisms which do not move maximally the generic type coincide with the bounded automorphisms  and thus build a normal subgroup. This will allow us to extend Corollary \ref{C:main} in Remark \ref{R:unbounded_movemax2}. 

\section{General model-theoretic properties yield our properties}\label{S:check}

In this short section, we will show that most of our properties can be easily verified for suitable simple theories as long as they satisfy some general model-theoretic properties. This will be relevant in order to show that several of our examples of Theorem \ref{T:mainExamples} (see Theorem B of the Introduction) fit into our framework. 

As before, we fix a cardinal $\kappa \ge |\LL|^+$ and a $\kappa$-saturated model $\M$ of a simple $\LL$-theory  $T$.

\begin{lemma}\label{L:stable_Konnerth}\textup{(}\cf\ \cite[Lemma
	2.3]{rK02}\textup{)}~
	Suppose that  $T$ is a stable theory satisfying Property \ref{H:single_gen}. Moreover, we assume that every type over an 
	algebraically closed subset of  $\M$ is stationary. (The latter always holds if $T$ has weak elimination of imaginaries.) Then 
	$T$ satisfies Properties
	\ref{H:closed_stat} and (WH) 
	(see Definition \ref{D:Konnerth}) with respect to the closure operator defined in   Definition \ref{D:closure}.
\end{lemma}

\begin{proof}
	
	Property   \ref{H:closed_stat} always holds in a  theory 
	for which types over algebraically closed subsets are stationary.  Thus, we need 
	only show that $\M$ is $\kappa$-atomic over any subset $K$ of the form $K=\acl(\cl(A_1)\cup\cdots\cup\cl(A_n)\cup B)$, where all $A_i$'s and $B$ have size strictly less than $\kappa$.
	In 
	order to show
	that $\tp(c/K)$ is $\kappa$-isolated, where $c$ is a finite tuple of $\M$, we must find some
	subset $E$ of $K$ of size strictly less than $\kappa$ such that
	$\tp(c/E)\models \tp(c/K)$.
	The local character of forking  yields  a  subset $C$ of $K$ of size at most $|\LL|$ such that $c\ind_{C} K$.
	Set now  \[ E =\acl(A_1,\cdots, A_n, B, C).\] Since $C\subseteq E$, it
	follows that $c\ind_E K$.
	
	We now show that $\tp(c/E)\models \tp(c/K)$,  or equivalently, 
	that
	$\tp(c/E)\models \tp(c/E, \eta)$ for all finite tuples $\eta$ in 
	$K$.
	Since $\tp(c/E)$ is stationary,  we need only show that for every realization $d$ in
	$\M$ of
	$\tp(c/E)$  \[ d\ind_E \eta.\] By $\kappa$-saturation, there is a tuple $\eta'$
	in $\M$ such that $c\eta'\equiv_E d\eta$. It suffices thus to show
	that $c\ind_E \eta'$. Now, the tuple $\eta$ is algebraic over  $E\cup\{
	\eta_i\}_{1\le i\le n}$, where each $\eta_i$ belongs to $\cl(A_i)$ for
	$1\le i \le n$. Hence,  the tuple $\eta'$ is algebraic over  $E\cup\{
	\eta'_i\}_{1\le i\le n}$, where each $\eta'_i$ again belongs to
	$\cl(A_i)$ for $1\le i \le n$, since $\cl(A)$ is invariant under all automorphisms of $\M$ fixing $A$.  We deduce that $\eta'$ lies in $K$, as desired.
\end{proof}

Recall that a stationary type $q$ over $\emptyset$ is \emph{regular} if it is  orthogonal (see Definition \ref{D:orth}) to every forking extension of $q$, that is, if $a$ realizes the unique nonforking extension of $q$ to $B$ and the realization $c$ of $q$ is not independent from $B$, then $a\ind_B c$. A straight-forward application of the inequalities of Lascar yields that types with Lascar rank  of the form $\omega^\gamma$ for some  ordinal $\gamma$ are always regular. In particular, the generic type of a strongly minimal set is always regular. 

\begin{prop}\label{P:gen_reg_addgen}
	Assume now that the simple theory $T$ satisfies Property \ref{H:single_gen} with respect to the generic type $p_0$, which we assume to be regular. The following hold:
	\begin{enumerate}[(a)]
		\item\label{I:gen=not_in_cl} If the realization $b$ of $p_0$ is not generic over the subset $A$, then $b$ belongs to $\cl(A)$. (The converse is also true and follows immediately from Remark \ref{R:closure} (b))
		\item\label{I:prime_tame} Property \ref{H:addgen} holds. 
		\item\label{I:fte} The conclusion of Remark \ref{R:addgen_weaker_version} (\ref{I:addgen_old}) holds with $A_0$ a finite sequence. 
		\item\label{I:asm} If
		$T$ is almost strongly minimal, then the algebraic closure $\acl(A)$ of a subset $A$ of $\M$ equals the closure $\cl(A)$ of $A$. 
		\item\label{I:gp_noniso}  
		If the universe of a model of $T$ is a group with a single generic stationary type as in 
		Remark \ref{R:exemples_base}, then the closure of the set $A$ equals \[
		\cl(A)=\{b\in \M \ | \ b \text{ is not generic over } A \}.\] 
	\end{enumerate}   
\end{prop}

\begin{proof}
	For (\ref{I:gen=not_in_cl}), consider $C\supset A$ and  an 
	element $g$ generic over $C$. Now, the realization $b$ is not 
	generic over $C$, so $g\ind_C b$ by regularity. Hence, we 
	conclude that $b$ belongs to $\cl(A)$, as desired.

	For (\ref{I:prime_tame}), given a $\kappa$-saturated model $\M$ of $T$ and a subset $\paraset$ of parameters of size strictly less than $\kappa$, we need to show that $\M=\cl(\paraset \cup A)$, where $A$ enumerates an independent sequence realizations of the generic  type over $\paraset$. By Property \ref{H:single_gen}, every $b$ in $\M$ is algebraic over finitely many realizations $h_1,\ldots, h_n$ of the generic type. It suffices hence to show that each $h_i$ belongs to $\cl(\paraset \cup A)$, which follows immediately from the above discussion, since $h_i$ is not generic over $\paraset \cup A$, by maximality of $A$. 
	
	For (\ref{I:fte}), given a set of parameters $\paraset$ of size strictly less 
	than $\kappa$ and an element $b$, choose by Property 
	\ref{H:single_gen} a finite set of realizations $h_1,\ldots, h_n$ 
	of $p_0$ algebraizing $b$. Up to reordering, we may assume that 
	$h_1,\ldots, h_m$, with $m\le n$, is a maximal subtuple 
	independent over $\paraset$, so $h_r \nind_\paraset h_1,\ldots, h_m$ for 
	$m+1\le r\le n$ (If $m=0$, then the latter independence means 
	that no $h_r$ is generic over $\paraset$).  By (\ref{I:gen=not_in_cl}), each $h_r$, and hence $b$, belongs to $\cl(\paraset\cup A_0)$, where $A_0=\{h_1,\ldots, h_m\}$, as desired.  
	
	For (\ref{I:asm}), it suffices to show that every element $b$ of $\cl(A)$ is algebraic over $A$. By Property \ref{H:single_gen}, the element $b$ is algebraic over finitely many realizations $h_1,\ldots, h_n$ of the strongly minimal type $p_0$. Choosing a maximal $A$-independent subtuple and relabelling, we may assume that $b$ belongs to $\acl(A, h_1,\ldots, h_n)$ and choose $n$ minimal such. If $n\ne 0$, notice that $h_n$ is generic over $A, h_1,\ldots, h_{n-1}$, so \[ h_n\ind_{A, h_1,\ldots, h_{n-1}} b,\] for $b$ belongs to $\cl(A)$. We conclude that $b$ belongs to $\acl(A, h_1,\ldots, h_{n-1})$, which gives the desired contradiction. 
	
	For (\ref{I:gp_noniso}), we need only show that if the element $b$ of $\M$ does not belong to $\cl(A)$, then it must be generic over $A$. Choose now $g$ 
	generic over $A\cup\{b\}$ and write $b=g\cdot h$, with 
	$h=g\inv\cdot b$ generic over $A$. Now, the element $h$ is a realization of 
	$p_0$. If $h$ were not generic over $A\cup\{g\}$, then by 
	(\ref{I:gen=not_in_cl}), the element $h$, and thus $b$, would belong to $\cl(A\cup\{g\})$, and thus to $\cl(A)$, by 
	Remark \ref{R:closure} (c), which is a contradiction. Hence, the element $h$ must be generic over $A\cup\{g\}$, and thus so is $b$ generic over $A$, as desired.  
\end{proof}

Lemma \ref{L:stable_Konnerth} and Proposition \ref{P:gen_reg_addgen} yield the following consequence. 
\begin{cor}\label{C:Premiers_Examples}~
	
	Every stable connected group with weak elimination of imaginaries whose
	generic type is regular satisfies Properties
	\ref{H:single_gen}-\ref{H:addgen} as well as (WH).
	If the group is superstable of Lascar rank $\omega^\alpha$, then the
	closure operator equals \[ \cl(A)=\{b\in
	\M \ | \ \mathrm{U}(b/A)<\omega^\alpha\} .\]        
\end{cor}
\begin{proof}
	Property \ref{H:single_gen} follows since the group is connected, by Remark \ref{R:exemples_base}. Properties
	\ref{H:closed_stat} and
	(WH) follow from Lemma \ref{L:stable_Konnerth},
	whilst Property \ref{H:addgen} follows from Proposition
	\ref{P:gen_reg_addgen} (\ref{I:prime_tame}) and (\ref{I:fte}). The description of the closure is Proposition \ref{P:gen_reg_addgen} (\ref{I:gp_noniso}).
\end{proof}

We recall now \cite[Definition 2.14]{BHMP17} (or a slightly modified
version thereof, see also \cite{fW20}). We will hence assume that the simple theory $T$ satisfies Property \ref{H:single_gen} with respect to the generic type $p_0$. 
\begin{definition}\label{D:bounded}
	An $\LL$-automorphism $\rho$ of a $\kappa$-saturated
	model $\M$ is \emph{bounded} if there exists some subset $A$ of cardinality strictly less than $\kappa$ such that for every $m$ in $\M$ we have that $\rho(m)$ belongs to $\cl(A\cup\{m\})$. We say that $\rho$ is \emph{unbounded} if it is not bounded.
\end{definition}

We may assume that $A$ is stable under the action of $\rho$ in the above definition of bounded, since $\kappa$ is uncountable.  It is immediate to see that the collection of bounded 
automorphisms of an arbitrary model forms a normal subgroup of the
automorphism group of $\UU$.

Whenever we work with a theory of fields, the Frobenius map is a bounded automorphism if the underlying field has  positive characteristic $p$ and is perfect. However, no power of Frobenius is the identity on  $\cl(\emptyset)$, as long as $\cl(\emptyset)$ contains $\overline{\mathbb F_p}$.

\begin{remark}\label{R:unbounded_movemax}~
	\begin{enumerate}[(a)]
		\item No bounded automorphism of $\M$ moves $p_0$ maximally.
		\item The converse is true whenever $\cl(A)=\{b\in \M \ | \ b $ is not generic over $A \}$, which is always the case by Proposition \ref{P:gen_reg_addgen} if $T$ is strongly minimal or if the universe of every model of $T$ is a group whose unique generic type is stationary and regular. 
	\end{enumerate}
	In particular, for both of the cases listed in (b), if there is no non-trivial bounded automorphism fixing $\cl(\emptyset)$, then Property \ref{H:nobounded} holds. 
\end{remark}

\begin{proof}
	For (a), assume that the automorphism $\rho$ is bounded over the $\rho$-invariant subset $A$. For every generic element $b$ over $A$, we have that $\rho(b)$ belongs to $\cl(A\cup\{b\})$, so $b$ and $\rho(b)$ cannot be independent over $A$ by Remark \ref{R:closure} (c). Thus, the automorphism $\rho$ does not move maximally the generic type. 
	
	Assume now for (b) that the closure of every subset $A$ is the collection of non-generic elements of $\M$ over $A$. If $\rho$ is unbounded, then  for every $\rho$-invariant subset $A$ of cardinality strictly less than $\kappa$, there is some element $m$ in $\M$ such that $\rho(m)$ does not belong to $\cl(A\cup\{m\})$ and thus $\rho(m)$ does not belong to $\cl(A)$, so $\rho(m)$, and hence $m$, is generic over $A$, by our assumption. Moreover, the element $\rho(m)$ is generic over $A\cup\{m\}$, so $\rho$ moves $p_0$ maximally. 
\end{proof}

\begin{remark}\label{R:unbounded_movemax2}
	If the bounded automorphisms coincide with those automorphisms which do not move maximally the generic type (as it is the case for the theories listed in Remark \ref{R:unbounded_movemax}), then the proof of Theorem \ref{T:main} shows that the quotient group $\Aut(\UU/\cl(\paraset))/N$ is simple, where $N$ is the normal subgroup of all automorphisms in $\Aut(\UU/\cl(\paraset))$ which are bounded (see \cite[Th\'eor\`eme 2]{dL92}).  
\end{remark}
\section{Annex: The examples}

Using the results of the previous sections, we will now show
that all the theories listed in Theorem B all satisfy the
assumptions of Corollary \ref{C:main}. We will conclude the section with some remarks and open questions.

\subsection*{The two classical $1$-based strongly minimal theories}

Algebraic closure defines a natural pregeometry on a strongly minimal set, and thus each subset of  the strongly minimal set $X$ has a basis, and therefore a \emph{dimension}. The $\emptyset$-definable strongly minimal set $X$ 
is locally modular if the dimension formula \[ \dim(A\cup B) =\dim(A)+\dim(B) -\dim(A\cap B)\] always holds for every two algebraically closed finite-dimensional subsets $A$ and $B$ with $\dim(A\cap B)>0$. Archetypal examples of locally modular strongly minimal sets are infinite subsets with no additional structure as well as infinite dimensional vector spaces over a fixed division ring. In the case of infinite sets, the dimension is just the cardinality of the set whilst for vector spaces, it is the linear dimension of the subspace generated. 

If $\kappa>|\LL|$, the two classical examples are $\kappa$-categorical and thus every model $\UU$ of cardinality $\kappa$ is saturated. Remark \ref{R:exemples_base} yields that these two examples satisfy Property \ref{H:single_gen}. Now, the algebraic closure coincides with the definable closure, so it follows immediately that every type is stationary, so Property \ref{H:closed_stat} and (WH) hold by Lemma \ref{L:stable_Konnerth}. Moreover, Property \ref{H:addgen} holds trivially.  Lemma \ref{L:sat_tame} yields now that the saturated model $\UU$ is $\kappa$-tame over $\emptyset$.  By Proposition \ref{P:gen_reg_addgen} (\ref{I:asm}) and Remark \ref{R:unbounded_movemax}, we conclude by Remark \ref{R:unbounded_movemax2} that the quotient $\Aut(\UU)/N$ is simple, where $N$ is the normal subgroup of all automorphisms in $\Aut(\UU)$ which are bounded. 

In the particular case of an infinite set, an automorphism is a permutation. Since the closure is totally disintegrated, it follows that a permutation $\rho$ is bounded if and only if its support $\{x\in \UU \ | \ \rho(x)\ne x\}$ has cardinality strictly less than $\kappa$. Automorphisms of vector spaces are linear isomorphisms. If for the isomorphism $\rho$ there is some scalar $\lambda$  such that $\UU_\lambda=\mathrm{Ker}(\varphi-\lambda\cdot \mathrm{Id}_\UU)$ has codimension strictly less than $\kappa$, it is easy to see that $\rho$ is bounded over $A=W\cup \rho(W)$, where $W$ is some linear complement of $\UU_\lambda$. Conversely, if $\rho$ is bounded over the subspace $A$, which we assume to be stable under $\rho$ and $\rho\inv$, it is easy to show that there exists some scalar $\lambda$ such that $\rho(b)-\lambda\cdot b=a$ belongs to $A$ for every $b$ in $\UU$.  Enlarging $A$ by a set (and their images by $\rho$) of size strictly less than $\kappa$ of representatives $(b_a)_{a\in A}$ for a suitable choice of such $a$'s, we conclude that   $\mathrm{Ker}(\rho-\lambda\cdot \mathrm{Id}_\UU)$ has codimension strictly less than $\kappa$, as desired. 

Since these two theories are uncountably categorical and $\kappa>|\LL|$, the $\kappa$-prime model over $\acl(Z)$, where $Z$ is of cardinality strictly less than $\kappa$, is the unique saturated model of size $\kappa$ containing $Z$, so for the sake of the presentation, we will assume that $Z$ has been named as constants to the language. Hence,  we recover the results of Baer \cite{rB34} and Rosenberg \cite{aR58} mentioned at the beginning of the Introduction. 

\begin{prop}\label{P:classical_sm}
	Given an uncountable cardinal $\kappa$ and an infinite set $\mathcal M$ of cardinality $\kappa>|\LL|$ with no additional structure (but possibly some additional constants being named), the group of permutations of $\mathcal M$ is simple modulo the subgroup of permutations of support strictly less than $\kappa$. 
	
	Similarly, given a field $K$ and an  uncountable cardinal $\kappa>|K|$, the group of isomorphisms of a $K$-vector space $V$ of cardinality $\kappa$ is simple modulo the subgroup of those isomorphisms $\varphi$ such that for some $\lambda$ in $K$, the eigenspace $\mathrm{Ker}(\varphi-\lambda\cdot \mathrm{Id}_V)$ has codimension strictly less than $\kappa$. 
\end{prop}

\subsection*{The almost strongly  minimal case}
As already introduced in Remark \ref{R:exemples_base}, a theory $T$ is \emph{almost strongly minimal} if there is a
strongly minimal set $X$ defined over $\emptyset$, such that  every model 
$\M$ of $T$ equals $\M=\acl(X(\M))$. The unique
non-algebraic type of $X$ is the \emph{generic} type and is regular. From now on, we will assume that $\M$ is $\kappa$-saturated with $\kappa\ge |\LL|^+$. 

Proposition \ref{P:gen_reg_addgen}(\ref{I:asm}) yields that $\cl(A)$ equals $\acl(A)$. 
Almost strongly minimal theories are again uncountably categorical, so the $\kappa$-prime model over $\acl(Z)$, where $Z$ is of cardinality strictly less than $\kappa$, is the unique saturated model of size $\kappa$ containing $Z$. 

If $X$ is the strongly minimal $\emptyset$-definable set $X$ associated to $\M$ with $X\cap \acl(\emptyset)$ infinite,  a well-known result of Lascar and Pillay yields that  every algebraically closed subset is an elementary substructure (which immediately gives Property \ref{H:closed_stat}). In particular, types over algebraically closed subsets will  be  stationary. The latter can always be achieved by working in $T^{eq}$, but we will prefer not to change our language: even if there is a natural isomorphism between the automorphism group of $\M$ and the automorphism group of $\M^{eq}$, it need not be the case that the corresponding groups $\Aut(\M/\acl(\emptyset))$ and $\Aut(\M^{eq}/\acl^{eq}\emptyset))$ are isomorphic.

\begin{prop}\label{P:asm_prop} 
	
	Every almost strongly minimal theory $T$ with respect to the strongly minimal $\emptyset$-definable set $X$ such that $X\cap \acl(\emptyset)$ is infinite satisfies Properties
	\ref{H:single_gen}-\ref{H:addgen} as well as (WH).
	
	Therefore, for every subset $\paraset$ of parameters of cardinality strictly less than $\kappa$,  every (saturated) model $\UU$ of size $\kappa$  is 
	$\kappa$-tame over $\paraset$. In particular, the conjugacy class in the group $\Aut(\UU/\acl(\paraset))$ of an automorphism  which moves $p_0$ maximally generates the whole group in four steps.   
\end{prop}
\begin{proof}
	Properties \ref{H:single_gen} and \ref{H:addgen}  follow from Remark \ref{R:exemples_base} as well as Proposition
	\ref{P:gen_reg_addgen} (\ref{I:prime_tame}) and (\ref{I:fte}). In order to show Property \ref{H:closed_stat} and (WH), we need only show that types over algebraically closed subsets are stationary by Lemma \ref{L:stable_Konnerth}. 
	
	We will show that for almost strongly minimal theories with $X\cap \acl(\emptyset)$ infinite, every algebraically closed subset $\acl(A)$ is an elementary substructure. This is probably well-known, but we could not find a suitable reference. It is a straight-forward application of Tarski's test: consider a realization $b$ of a formula $\varphi(x, c_1,\ldots, c_n)$, where each
	$c_i$ belongs to $\acl(A)$. Now, the element $b$ is algebraic over a sequence $d_1,\ldots, d_m$ of realizations of the strongly minimal set $X$. Up to relabelling, we may assume that the $d_i$'s are $A$-independent, so $b$ belongs to $\acl(A, d_1,\ldots, d_m)$. Choose $m$ least possible such there exists such a realization of $\varphi$. If $m=0$, then we are
	done. Otherwise, choose a formula $\psi(x, d_1,\ldots, d_m)$ with parameters in $A$ which holds for $b$ and such that $\psi(x, e_1,\ldots, e_m)$ is always algebraic for all $(e_1,\ldots, e_m)$ in $X^m$. 
	
	Since $X\cap \acl(\emptyset)$ is infinite, the generic type $p_0$ cannot be isolated. Thus, by the Open
	Mapping theorem, neither is the type $tp(d_m/\acl(A), d_1,\ldots, d_{m-1})$. The formula \[ X(y)\land \exists x\big(\varphi(x, c_1,\ldots,c_n)\land \psi(x, d_1, \ldots, d_{m-1}, y) \big) \] 
	must therefore contain a realization $e$ which is not generic, and thus must be algebraic, over $A, d_1,\ldots, d_{m-1}$,  giving the desired contradiction.   
	
	Therefore, using Remark \ref{R:Exist}, Proposition \ref{P:prime_tame} and Theorem \ref{T:main}, we deduce the desired conclusion of the statement. 
\end{proof}

Lascar showed in \cite[Proposition 14]{dL92} that, for a $1$-based strongly minimal set $\UU$, given a subset $\paraset$ of $\UU$ of cardinality strictly less than $\kappa$, the group of automorphisms $\Aut(\UU/\acl(\paraset))$ always contains (many) non-trivial bounded elements. However, this is not the case for the non-locally modular strongly minimal theory ACF$_p$ of algebraically closed fields of characteristic $p$, for $p$ either $0$ or a prime, see \cite[Lemma 2]{dL97} and \cite[Theor\`eme 3.1]{BHMP17}.

A straight-forward application of Remark \ref{R:unbounded_movemax}, together with Proposition \ref{P:asm_prop} and Corollary \ref{C:main}, yields the following: 
\begin{cor}\label{C:ACF}
	Every strongly minimal theory ACF$_p$ of algebraically closed fields of characteristic $p$, for $p$ either $0$ or a prime satisfies that there are no non-trivial bounded automorphisms fixing $\cl(\emptyset)$ pointwise, so Property \ref{H:nobounded} holds.
	
	Therefore, for every subset $Z$ of cardinality at most $\kappa$ inside an algebraically closed field $\UU$ of cardinality $\kappa$,  the group $\Aut(\UU/\acl(\paraset))$ is simple.
\end{cor} 

\subsection*{Differential fields}

We fix some natural number $m\ge 1$ and consider first the  theory DCF$_{0,m}$ of differentially closed fields of characteristic $0$ with
$m$ commuting derivations (see \cite{tMG00} for the general 
results as well as \cite[Chapter II]{MMP96} for $m=1$). The theory DCF$_{0,m}$ is complete,  eliminates
quantifiers and imaginaries and is  $\omega$-stable. The  generic type over $\emptyset$ is the type of a differentially transcendental element, that is, satisfying no non-trivial differential equation over $\Q$. Note that this type is stationary and its Lascar  rank equals $\omega^m$, so the generic type is regular. Since the underlying additive group is divisible, it is connected. By Corollary \ref{C:Premiers_Examples}, the theory DCF$_{0,m}$  satisfies Properties
\ref{H:single_gen}-\ref{H:addgen} as well as (WH). Moreover, if $A$ is a subset of a $\kappa$-saturated model of DCF$_{0,m}$, then $\cl(A)$ is the field of elements satisfying some non-trivial differential equation over the differential field generated by $A$. 

It was shown in \cite[Theorem 3.1]{BHMP17} that the only bounded automorphism of a $\kappa$-saturated model of DCF$_{0,m}$ is the identity. By Remark \ref{R:unbounded_movemax}, Property \ref{H:nobounded} holds for DCF$_{0,m}$. 

Another theory  of differential fields to consider was obtained by Hrushovski and Itai in \cite{eHI}.  The models of the theory $T(X)$ are the existentially
closed models $K$ of the class of differential fields of characteristic $0$ in which the equations
defining $X:=\{x\in C \mid Dx=s(x)\}$ have no
solution, where $C$ is a curve of genus at least $1$  and $s:C\to T(C)$ is a rational section of the tangent bundle, all defined over the field of constants of $K$. Each theory $T(X)$ is complete, eliminates quantifiers and imaginaries, and is $\omega$-stable \cite[Proposition 4.1 \& Lemma 4.6]{eHI}. As before, the generic type $p_0$ is the stationary type of a differentially
transcendental element, and has Lascar rank $\omega$. The underlying additive group is again connected and $\cl(A)$ is the field of elements satisfying some non-trivial differential equation over the differential field generated by $A$. In particular,  each theory $T(X)$ satisfies Properties
\ref{H:single_gen}-\ref{H:addgen} as well as (WH), by Corollary \ref{C:Premiers_Examples}.

The proof of  \cite[Theorem 3.1]{BHMP17} goes through verbatim for models of $T(X)$, so the only bounded automorphism of a $\kappa$-saturated model of $T(X)$ is the identity. Again by   Remark \ref{R:unbounded_movemax}, we have that Property \ref{H:nobounded} holds for $T(X)$.

As before, we deduce from Proposition \ref{P:prime_tame} and Corollary \ref{C:main} the following result:

\begin{prop}\label{P:DCF} Let $T$ be either the theory DCF$_{0,m}$ of differentially closed fields of characteristic $0$ with
	$m$ commuting derivations  or one of the theories $T(X)$ described above. For every subset $\paraset$ of parameters of cardinality strictly less than $\kappa$, every saturated model $\UU$ of size $\kappa$ of $T$ as well as every $\kappa$-prime model over $\acl(\paraset)$ is 
	$\kappa$-tame over $\paraset$.  For every $\kappa$-tame model $\UU$ over $\paraset$, the group $\Aut(\UU/\cl(\paraset))$ is simple, where 
	\begin{multline*} \cl(\paraset) = \{ x\in \UU \ | \ x \emph{ satisfies some  non-trivial differential equation}  \\ \emph{ over the differential field generated by } \paraset \}.\end{multline*} 
\end{prop} 

\subsection*{Difference fields of characteristic $0$}

Recall that ACFA$_0$ denotes the theory of existentially closed
difference fields $(K,\sigma)$ of characteristic $0$. Every model $\M$ of ACFA$_0$ is algebraically closed (as a field) and \emph{inversive} (that is, the endomorphism $\sigma$ is surjective). If $A\subset \UU$, then $\acl(A)$ is the smallest
algebraically closed inversive difference field containing $A$. The type $\tp(b/A)$ of a a tuple $b$ in $\M$ is entirely determined by the isomorphism type
over $A$ of $\acl(Ab)$ (see \cite[Corollary 1.5]{CH99}). Every
completion of ACFA$_0$ eliminates imaginaries (\cite[(1.10)]{CH99}) and is simple. Moreover, the    nonforking independence can be described in algebraic terms: Given two
inversive difference fields $A=\acl(A)$ and $B=\acl(B)$ with a common difference
subfield $C=\acl(C)$, we have that $A\ind_C B$ if and only if
$A$ and $B$ are algebraically independent over $C$ (see
\cite[Section 
(1.9)]{CH99}).

\begin{remark}\label{R:ACFA0}
	Every completion $T$ of ACFA$_0$ is simple and satisfies Properties
	\ref{H:single_gen}-\ref{H:nobounded}.  The closure operator can be described algebraically as \[ \cl(B)=\{a\in
	\UU \ | \ \mathrm{tr.deg}(a, \sigma(a), \sigma^2(a),\ldots/\acl(B))<\infty\}.\] 
\end{remark}
\begin{proof}
	For Property \ref{H:single_gen}, notice that there is a unique generic $1$-type $p_0$ (generic in the sense of the field structure), which says that its realization does not satisfy any non-trivial difference equation. This type is stationary (\cite[Proposition 2.10]{CH99}) and is the only $1$-type of Lascar rank $SU(p_0)=\omega$, so it is regular. By Remark \ref{R:exemples_base}, the theory ACFA$_0$ satisfies the hypothesis of Proposition \ref{P:gen_reg_addgen}, hence Property \ref{H:addgen} holds. Property \ref{H:closed_stat}  was already shown in
	\cite[Theorem 5.3 and Corollary B.11]{CH04}.  Remark \ref{R:unbounded_movemax} yields Property \ref{H:nobounded} using \cite[Theorem 3.1]{BHMP17}. 
	
	The description of the closure follows now directly from Proposition \ref{P:gen_reg_addgen} (\ref{I:gp_noniso}) and the above characterisation of the generic type. 
\end{proof} 

The theory ACFA$_0$ is unstable, so we do not know whether there are saturated models in their cardinality as suitable candidates for $\kappa$-tame models in order to apply the results of Section \ref{S:Lascar}. Now, given an uncountable cardinal $\kappa$, the existence and  uniqueness of a $\kappa$-prime model $\UU$ over 
$\cl_\UU(\paraset)$ was shown by the second author
\cite{zC23}: Choose a $\kappa$-saturated
model $\mathcal M$ of ACFA of characteristic $0$ containing a subset $\paraset$ of cardinality strictly less than $\kappa$. Then $\kappa$-prime
models over algebraically closed difference subfields $A$ of $\M$ containing $\cl_{\M}(\emptyset)$ exist and are unique up to isomorphism over $A$ (see  Theorem 3.17 in \cite{zC23}).  As in the stable case, the $\kappa$-prime model $\UU$ of ACFA$_0$ over 
a subset $A$ of $\M$ containing $\cl_{\M}(\emptyset)$
is characterised by being $\kappa$-saturated and
containing $A$,  $\kappa$-atomic 
over $A$ and containing no (non-constant) $A$-indiscernible sequence of
length     $\kappa^+$ (Theorem 3.14 in \cite{zC23}). In particular, every completion of the theory ACFA$_0$ satisfies the condition of Proposition \ref{P:prime_tame}.

Therefore, the $\kappa$-prime $\UU$ over $\cl_{\M}(\paraset)$ will be $\kappa$-tame over $\paraset$ once we show that Property (WH) holds for ACFA$_0$, by Proposition \ref{P:prime_tame}. For the sake of the presentation, we have decided to provide two alternative proofs of this, one   along the lines of the proof of Lemma \ref{L:stable_Konnerth}
and another one (see Remark \ref{R:proof_Zoe}) using the strength of the semi-minimal analysis of types in ACFA$_0$.

\begin{prop}\label{P:ACFA_Konnerth}
	Every completion $T$ of ACFA$_0$ satisfies Property (WH).
\end{prop}
\begin{proof}
	In order to show that the $\kappa$-saturated model $\M$ is $\kappa$-atomic over any subset $K$ of the form $K=\acl(\cl(A_1)\cup\cdots\cup\cl(A_n)\cup B)$, where all $A_i$'s and $B$ have size strictly less than $\kappa$, we need to show, as in the proof 
	of Lemma
	\ref{L:stable_Konnerth},  that  $\tp(c/E)\models \tp(c/E\cup\{\eta\})$ for every finite tuples $c$ of $\M$ and $\eta$ of $K$, with $E=\acl(A_1,\cdots, A_n, B, 
	D)$, 
	where  $D$ is a subset of
	$K=\acl(\cl(A_1)\cup\cdots\cup\cl(A_n)\cup
	B)$
	of cardinality  bounded by $|\LL|<\kappa$ such that $c\ind_{D} 
	K$. Note 		that $c\ind_E K$. Possibly at the
	cost of enlarging $\eta$, we may also assume that
	$\eta=(\eta_0, \eta_1, \ldots ,\eta_n)$ with $\eta_0$ field 
	algebraic
	over $E\cup\{ \eta_i\}_{1\le i\le n}$
	and $\eta_i$ in $\cl(A_i)$ for $1\le i \le n$. 
	
	By \cite[Corollary 1.13]{CH99}, the field  $(\M,\sigma^k)$,  is again
	a difference closed field for $k\ne 0$ in $\Z$. In particular, we will
	denote all throughout this proof by $\scl{A}_{\sigma^k}$, $\acl_{\sigma^k}(A)$, $\cl_{\sigma^k}(A)$ and
	$\tp_{\sigma^k}(d)$ the corresponding notions in the reduct  $\M[k]=(\M,\sigma^k)$.	
	\begin{claim*}
		If 
		$ \scl{E\cup\{c\}}_{\sigma^k}$ and 
		$                \scl{E\cup\{\eta'\}}_{\sigma^k} $  are algebraically
		independent over $E$ 
		for every  $k\ne 0$ in $\Z$ and every finite tuple $\eta'$ realizing
		the quantifier-free type $\mathrm{qftp}_{\sigma^k}(\eta)$, then the type
		$\tp_\sigma(c/E)$ implies $\tp_\sigma(c/E\cup\{\eta\})$. 
	\end{claim*}
	\begin{claimproof*}
		Assume that $c$, $\eta$ and $E$ are as in the statement. We can apply \cite[Proposition 4.9]{CH99} and obtain that $\tp_\sigma(c/E)$ and
		$\tp_\sigma(\eta/E)$ are superficially co-stable: for every $d\ind_E
		\eta$ realizing $\tp_\sigma(c/E)$,  setting
		$K=\acl_\sigma(E\cup\{\eta\})$ and considering  the unique extension $\sigma_F$ of $\sigma\restr{\acl_\sigma(E,d)}\otimes \sigma\restr{K}$ to the compositum field
		$F$ of $\acl_\sigma(E,d)$ and $K$, we have that $F$ has no proper finite Galois extension invariant under $\sigma_F$.
		
		Moreover,  setting $k=1$
		in our assumption, we deduce from the previous description of the
		non-forking independence in the simple theory ACFA that $c$ is independent of every realisation
		$\eta'$ in $\UU$ of $\tp_\sigma(\eta/E)$, or equivalently, that every
		realisation $d$ of $\tp_\sigma(c/E)$ is independent from  $\eta$ over
		$E$.
		
		A realization $d$ of $\tp_\sigma(c/E)$ yields an $E$-isomorphism of difference fields between  $\acl_\sigma(E\cup\{d\})$ and $\acl_\sigma(E\cup\{c\})$ mapping $d$ to $c$. We need to show that this isomorphism of difference fields extends to an $K$-isomorphism of difference fields between $\acl_\sigma(K\cup\{d\})$ and $\acl_\sigma(K\cup\{c\})$.
		
		By superficial co-stability (since every $d$ realizing $\tp_\sigma(c/E)$ is automatically independent from $K$ over $E$), there is a unique extension of the field automorphism
		$\sigma_F$ to the field algebraic closure of $F$, by \cite[Lemma 2.8]{CH99}.
		We conclude that there is an isomorphism of difference fields between the algebraic closure of $F$  and the
		difference field $\acl_\sigma(K\cup\{c\})$ over $K$ mapping $d$ to $c$, so $d\equiv_K c$ for every realization $d$ of $\tp_\sigma(c/E)$, as desired. 
	\end{claimproof*}
	By the above Claim, we need only show  that 
	for every
	integer $k$ in $\N$ and every $\eta'$ in $\UU$ realizing the
	quantifier-free type $\mathrm{qftp_{\sigma^k}(\eta/E)}$,
	the fields $\scl{E,c}_{\sigma^k}$ and 
	$\scl{E,\eta'}_{\sigma^k}$ are
	algebraically independent over $E$. Note that
	$\cl_\sigma(A_i)=\cl_{\sigma^k}(A_i)$, so $\eta_i$ belongs to
	$\cl_{\sigma^k}(A_i)$ for $1\le i\le n$. The  isomorphism of 
	difference
	fields in the structure $(\UU, \sigma^k)$ over $E$ maps
	the tuple $\eta$ to $\eta'$, with $\eta'=(\eta'_0,
	\eta'_1,\ldots,\eta'_n)$, where $\eta'_0$ is field algebraic over
	$E\cup\{ \eta'_i\}_{1\le i\le n}$ and each
	$\eta'_i$ lies in $\cl_{\sigma^k}(A_i)$, since each of these
	properties is quantifier-free definable in the language of difference
	rings. Hence, the whole tuple $\eta'$ is field algebraic over
	the field generated by $E\cup\bigcup_{1\le i\le n} 
	\cl_{\sigma}(A_i)
	\subseteq K$. The difference field $K$ is algebraically closed, so
	$\scl{E,\bar\eta'}_{\sigma^k}\subseteq K$.  Since $E$ was 
	chosen so
	that $c\ind_E K$, we deduce the desired algebraic independence between
	$\scl{E,c}_{\sigma^k}$ and $\scl{E,\eta'}_{\sigma^k}$ over 
	$E$. 
\end{proof}

\begin{remark}\label{R:proof_Zoe} We provide now an alternative proof that every completion $T$ of ACFA$_0$ satisfies Property (WH): Using the above notation, we want to show that $\tp_\sigma(c/K)$ is  $\kappa$-isolated. The proof is by induction on the Lascar
	rank $\SU(c/K)$ of $\tp(c/K)$,
	and we assume the property proved for any type of SU-rank
	$<\SU(c/K)$, over any set $K'$ of the same kind. If
	$\SU(c/K)=0$, there is nothing to prove. The proof
	distinguishes two cases: whether $\tp(c/K)$ is orthogonal to
	$\fix(\sigma)$, or whether it is not. 
	
	If $\tp(c/K)$ is orthogonal to $\fix(\sigma)$ then $\tp(c/K)$ is
	stationary, by \cite[Lemma 2]{Bu19}, and the proof given in Lemma \ref{L:stable_Konnerth} goes through verbatim.     
	
	Suppose now that 
	$\tp(c/K)$ is non-orthogonal to $\fix(\sigma)$. As
	$K$ contains $\fix(\sigma)(\M)$, every non-algebraic type
	over $K$ which is realised in $\M$ is weakly orthogonal to
	$\fix(\sigma)$. Hence, the difference field $\acl(Kc)$ has the same fixed field
	as $K$.
	
	The proof of \cite[Theorem 5.5]{CH99} for the case $q$
	non-orthogonal to $\sigma(x)=x$ (page 3049) gives that there is $b$ in $\acl(Kc)$ such
	that $\tp(b/K)$ is qf-internal to $\fix(\sigma)$. This
	means that there is some difference field $L$ containing $K$ and independent from $b$
	over $K$, such that $b$ belongs to $LF$.
	
	By \cite[Lemma 3.4]{zC23}, there is countable subset $D=\acl(D)$ of $K$ such that 
	$\tp(b/D)\vdash \tp(b/K)$,
	and we may impose that  $b$ belongs to $\acl(Dc)$. Thus $\tp(b/K)$ is
	$\kappa$-isolated. Our induction hypothesis applied to
	$\tp(c/\acl(Kb))$ yields that $\tp(c/\acl(Kb))$ is
	$\kappa$-isolated. We conclude that $\tp(b,c/K)$, and thus $\tp(c/K)$, is $\kappa$-isolated by 
	\cite[Remark 2.17 (1)]{zC23}. 
\end{remark}

We deduce from Remark \ref{R:ACFA0}, Propositions \ref{P:prime_tame} and \ref{P:ACFA_Konnerth} as well as  Corollary \ref{C:main} the following result:
\begin{prop}\label{P:ACFA}
	For every subset $\paraset$ of parameters of cardinality strictly less than $\kappa$, every $\kappa$-prime model over $\cl(\paraset)$ is 
	$\kappa$-tame over $\paraset$.  For every $\kappa$-tame model $\UU$ over $\paraset$, the group $\Aut(\UU/\cl(\paraset))$ is simple. 
\end{prop}

\subsection*{Proper pairs of algebraically
	closed fields}

The theory ACFP of proper pairs of algebraically
closed fields shares many traits of the theory DCF$_0$ (or rather DCF$_{0,1}$) of differentially closed fields of characteristic $0$, yet
it is somewhat simpler to describe. Most of the results mentioned here appear in \cite{jK64, Po83,  BPV03, MPZ20}. 

Every completion of ACFP in the language
$\LL_P=\mathcal{L}_\textrm{Rings}\cup\{P\}$, where $P$ denotes the
distinguished proper algebraically closed subfield $E$ of the model $K$ of ACFP  
is given by the characteristic of the field \cite{aR59}. The type of a subfield $k$ with
$k$ linearly disjoint from $E$ over $k\cap E$ (which we will denote by $k \ind^{ld}_{k\cap E} E$)  is uniquely determined by its quantifier-free
$\LL_P$-type, so ACFP has quantifier elimination after adding Delon's $\lambda$-functions \cite{fD12}, which play a similar role to the $\lambda$-functions for separably closed fields (see next subsection).  Given a tuple $a_0,\ldots, a_n$ of $K$, if  $a_1,\dotsc,a_n$ are
$E$-linearly dependent, then $\lambda_n^i(-;
a_1\dotsc,a_n)$ are equal to zero for $i=1,\ldots,n$. If $a_1,\dotsc,a_n$ are
$E$-linearly independent, but $a_0,a_1,\dotsc,a_n$ are not, then the  values of the $\lambda$-functions 
are uniquely determined by $a_0 = \sum\limits_{i=1}^n \lambda_n^i(a_0;
a_1\dotsc,a_n)\,a_i$. It follows from the above description of types that every completion of the theory ACFP is $\omega$-stable and that non-forking independence  can be characterised in terms of
two field independences \cite[Remark 7.2 and 
Proposition 7.3]{BPV03}:  Two $\LL_P$-definably closed
subfields $L_1$ and $L_2$ of a sufficiently saturated model $\M$ of ACFP  are
independent over a common $\LL_P$-definably closed subfield $k$ if and
only if \[ L_1\ind^{\rm ACF}_{k} L_2 \  \text{ and } \ L_1\ind^{\rm
	ACF}_{Ek} L_2,\] where $Ek$ denotes the subfield generated by
$E\cup k$ and $\ind^{\rm
	ACF} $ denotes independence in the sense of the reduct ACF.

\begin{remark}\label{R:ACFP}
	Every completion $T$ of ACFP is $\omega$-stable and satisfies Properties
	\ref{H:single_gen}-\ref{H:addgen} as well as (WH).  The closure operator can be described algebraically as $ \cl(A)= (EA)^{alg}$. 
\end{remark}
\begin{proof}
	For Property \ref{H:single_gen}, notice that there is a unique generic $1$-type $p_0$ (generic in the sense of the field structure) of (Morley and Lascar) rank $\omega$, whose realization is not a root of a non-trivial polynomial with coefficients in the subfield $E$. This type is stationary, by the above characterisation of independence in ACFP. Since $U(p_0)=\omega$, the type $p_0$ is regular. By Remark \ref{R:exemples_base}, the theory ACFP satisfies the hypothesis of Proposition \ref{P:gen_reg_addgen}, hence Property \ref{H:addgen} holds. Now, types over algebraically closed subsets in the theory ACFP are stationary, as shown by Bartnick \cite{cB24}, so we deduce  that Property \ref{H:closed_stat} and (WH) hold by Lemma \ref{L:stable_Konnerth}.
	
	The description of the closure follows now directly from Proposition \ref{P:gen_reg_addgen} (\ref{I:gp_noniso}) and the above characterisation of the generic type. 
\end{proof}

We will deduce now from  Remarks \ref{R:unbounded_movemax} and \ref{R:ACFP} as well as the following lemma that  every completion of ACFP satisfies Property \ref{H:nobounded}.
\begin{lemma}\label{L:bdd_pairs}
	Every bounded automorphism of a $\kappa$-saturated model $\M$ of the theory ACFP of pairs of algebraically closed field is the identity or a power of Frobenius (in positive characteristic). In particular, there is no non-trivial bounded automorphism fixing $\cl(\emptyset)=E$ and thus every completion of ACFP satisfies Property \ref{H:nobounded}.
\end{lemma}

\begin{proof}
	We will be concise, for the proof is a straight-forward adaptation of
	\cite[Theorem 3.1]{BHMP17}.  Assume that the automorphism $\rho$ of the
	$\kappa$-saturated pair $(\UU, E)$ is bounded over the $\rho$-invariant subset $A$. By
	the description of the $\cl$-closure, we know that $\rho(b)$ is algebraic over the subfield $EA(b)$ for every element $b$ of $\M$, in particular it is so for every generic element $b$ over $A$. 
	
	Choose now two generic independent elements $b_1$ and $b_2$ over
	$A$. By the description of the independence, the elements $b_1$, $b_2$ and  $b_1+b_2$ are pairwise algebraically independent over
	$EA$, and thus so are the pairs $(b_1, \rho(b_1))$, $(b_2,
	\rho(b_2))$ and $(b_1+b_2, \rho(b_1+b_2))$. Analogously, the pairs the pairs $(b_1, \rho(b_1))$, $(b_2,
	\rho(b_2))$ and $(b_1\cdot b_2, \rho(b_1\cdot b_2))$ are algebraically independent over $EA$, since $b_1$ and $b_2$ are also independent multiplicative generics.

	By Ziegler's lemma (\cite[Theorem 1, see also Lemma 2]{mZ06}, or
	\cite[Lemme 2.5]{BHMP17}), we
	deduce that there are two connected  algebraic subgroup $H_1$ of ${\mathbb G}_a^2$ and $H_2$ of $\mathbb G_m^2$, each defined over $(EA)^{alg}$,  such that $(b_1,\rho(b_1))$ is a generic element  of an additive translate of  $H_1$ and of a multiplicative translate of $H_2$, both translates defined over   $(EA)^{alg}$.  From here on, the rest of the proof of  \cite[Theorem 3.1]{BHMP17} goes verbatim and yields that $\rho$ is the identity or a power of Frobenius (in positive characteristic), as desired. 
\end{proof}
Remark \ref{R:ACFP} and Lemma \ref{L:bdd_pairs} together with Proposition \ref{P:asm_prop} and Corollary \ref{C:main}, together with 
Propositions  \ref{P:prime_tame} and \ref{P:asm_prop} as well as Corollary \ref{C:main} yield the following result.

\begin{cor}\label{C:ACFP}
	For every subset $\paraset$ of parameters of cardinality strictly less than $\kappa$, every saturated model $\UU$ of size $\kappa$ of ACFP as well as every $\kappa$-prime model over $\cl(\paraset)=(E\paraset)^{alg}$ is  $\kappa$-tame over $\paraset$.  
	
	For every $\kappa$-tame model $\UU$ of ACFP over $\paraset$, the group $\Aut(\UU/(E\paraset)^{alg})$ is simple.
\end{cor} 
The reader familiar with Lascar's original proof for the complex numbers will immediately notice that Corollary \ref{C:ACFP} follows from his proof, since an $\LL_P-$automorphism of $\UU$ fixing $E=P^{\UU}$ pointwise is just a field automorphism.

\subsection*{Separably closed fields of infinite degree 
	of imperfection}

Most of the references for this section appear  in
\cite{fD88} and \cite{Sr86}, unless explicitly stated.  From now on, we work inside an ambient field of positive characteristic $p>0$. 

\begin{definition}\label{D:SCF}~
	\begin{enumerate}
		\item Given two subfields $k\subset K$, the extension $k\subset K$ is \emph{separable} if $k\ind^{ld}_{k^p} K^p$. Given a subring $R$ of $K$, it is immediate to see that $k=\mathrm{Quot(R)}\subset K$ is separable if and only if every tuple of elements of $R$ which is $R^p$-linearly independent remains so over $K^p$.

		\item A subset $A$ of $K$ is \emph{$p$-independent (in $K$)} if for every $a$ in $A$, we have that $a$ does not belong to $K^p[A\setminus \{a\}]$. If $k\subset K$ is separable, a subset $A$ of $k$ is $p$-independent in $K$ if and only if it is $p$-independent in $k$. 
		
		\item  A subset $A$ of $K$ is a $p$-basis if it is maximal $p$-independent, or equivalently, if $A$ is $p$-independent and $K=K^p[A]$. More generally, given a separable extension $k\subset K$, a subset $B$ of $K$ is \emph{$p$-independent over $k$} if $A\cup B$ is $p$-independent in $K$, where $A$ is a $p$-basis of $k$. This is equivalent to requiring that no $b$ in $B$ belongs to $K^p[k\cup B\setminus \{b\}]$. 
	\end{enumerate}
\end{definition}
The \emph{degree of imperfection} of the field $K$ is the unique element $e$ of $\N\cup\{\infty\}$ with $[K:K^p]=p^e$. In particular, a field has infinite imperfection degree if and only if $K$ contains an infinite $p$-independent subset. 

The completions of the theory of separably closed fields of positive characteristic $p$ are uniquely determined by the imperfection degree \cite{yE68}. We will denote the corresponding completion by $SCF_{p,e}$ (notice that $SCF_{p,0}$ is the theory ACF$_p$). Ershov (see also Wood \cite{cW79}) provides a description of types: if $k$ and $k'$ are isomorphic subfields of a model $K$ of $SCF_{p,e}$ with both $k\subseteq K$ and $k'\subseteq K$ separable, then $k$ and $k'$ have the same type. In particular, the theory $SCF_{p,e}$ is stable \cite[Theorem 
3]{cW79}. It follows implicitly from the above description of types that if $k\subseteq K$ is separable, then $k$ is definably closed and its algebraic closure coincides with $k^{sep}$, the separable closure of $k$.

The above description of types yields an algebraic characterization of non-forking independence: Given two definably closed subsets $A$ and $B$ of a model $K$ of $SCF_{p,e}$ with $C=A\cap B$ algebraically closed, we have that \[ A\ind_C B \text{ if and only if } A\ind^{ACF}_C B \text{ and } AB\ind^{ld}_{A^p B^p} K^p,\] so $AB\subseteq K$ is separable, where $AB$ (resp. $A^pB^p$) denotes the field generated by $A\cup B$ (resp. $A^p\cup B^p$).

There is a simple way to expand the language in order to obtain quantifier elimination for $SCF_{p,e}$ after adding function symbols for the
$\lambda$-functions \cite[Remark 7]{Sr86}, defined analogously as in the case of pairs of algebraically closed fields (historically the $\lambda$-functions were first introduced for $SCF_{p,e}$). Given a tuple $a_0,\ldots, a_n$ of a model $K$ of $SCF_{p,e}$, if  $a_1,\dotsc,a_n$ are
$K^p$-linearly dependent, then $\lambda_n^i(x;
a_1\dotsc,a_n)$ is always zero. If $a_1,\dotsc,a_n$ are
$K^p$-linearly independent, but $a_0,a_1,\dotsc,a_n$ are not, the  values of the $\lambda$-functions 
are uniquely determined by $a_0 = \sum\limits_{i=1}^n \lambda_n^i(a_0;
a_1\dotsc,a_n)^p\,a_i$. Notice that every definably closed subset $A$ of $K$ is a subfield and the corresponding extension $A\subseteq K$ is separable. In particular, given a subset $A$ of a model $K$ of $SCF_{p,e}$, the definable closure $\dcl(A)$ is the smallest subfield of $K$ containing $A$ and closed under the $\lambda$-functions \cite[Lemma 0]{Sr86}. Moreover, the algebraic closure $\acl(A)$ is $\dcl(A)^{sep}$.

\begin{remark}\label{R:SCFP}
	The theory SCF$_{p,\infty}$ of separably closed fields of characteristic $p>0$ and infinite imperfection degree is stable and satisfies Properties
	\ref{H:single_gen} and \ref{H:closed_stat} as well as (WH).
\end{remark}
\begin{proof}
	We work inside a sufficiently saturated model $K$ of SCF$_{p,\infty}$. For Property \ref{H:single_gen}, notice that for any definably closed subfield $k$ and any element $g$ which is $p$-independent over $k$, the field extension $k(g)\subset K$ is again separable. The above characterization of types yields that any two $p$-independent $g$ and $g'$ elements over  $k$ have the same type over $k$. 
	We denote by $p_0$ a generic type of the (additive group of the) field $K$. If a realisation $g$ of $p_0$ were not $p$-independent over $\mathbb F_p$, then $g$ would be contained in $K^p$. Now, a generic type in a stable group only contains formulae which are generic (or \emph{syndetic}), so the subgroup $K^p$ would have finite index, which is a contradiction.  Hence $p_0$ is the unique generic type and the same argument yields that $p_0$ is stationary, since any two generic elements $g$ and $g'$ over the definably closed subfield $k$ are $p$-independent over $k$ and hence have the same type by the previous discussion. Remark \ref{R:exemples_base} now yields Property \ref{H:single_gen}. Moreover, an element $g$ realizes the non-forking extension of $p_0$ over $k$ if and only if $g$ is $p$-independent over $k$.  
	
	Now, types over algebraically closed subsets in the theory SCF$_{p,\infty}$ are stationary, as shown by Bartnick \cite{cB24}, so we deduce  that Property \ref{H:closed_stat} and (WH) hold by Lemma \ref{L:stable_Konnerth}.
\end{proof} 
\begin{remark}\label{R:SCF_noProp3}
	It is easy to see that the generic type $p_0$ of the theory SCF$_{p,\infty}$ not regular: Choose two independent realizations $g$ and $h$ of $p_0$. By Remark \ref{R:SCFP}, the element $g^p + h$ is again $p$-independent over the prime field, so it realizes $p_0$. Moreover, the type $\tp(g/\mathbb{F}_p(h))$ is a non-forking extension of $p_0$, whilst the type $\tp(g^p+h/\mathbb{F}_p(h))$ is a forking extension, since the additive subgroup $K^p$ does not have finite index in $K$. Now, the elements $g$ and $g^p+h$ are dependent over $\mathbb{F}_p(h)$, since $g$ is $p$-dependent over $\mathbb{F}_p(h, g^p+h)$.  
	
	The same proof works for the (unique) generic additive type of the field in any degree of imperfection $e\ne 0$.
\end{remark}

\begin{lemma}\label{L:SCF_closure}
	Given a subset $A=\acl(A)$ of a $\kappa$-saturated model $K$ of SCF$_{p,\infty}$,  the closure
	(with respect to $p_0$) is
	\[\cl(A)=\{ b \in K \mid \acl(Ab) \text{ has the same
		$p$-basis as } A \} 
	=\bigcap_{m\in \N} K^{p^m}[A].\]
	In particular, if $B$ is a $p$-basis of $K$ over $k$, then $\cl(k\cup B)=K$, so  SCF$_{p,\infty}$ satisfies Property \ref{H:addgen}. 
\end{lemma}

\begin{proof}
	Clearly, if $b$ belongs to $\cl(A)$, then every generic element $g$ over $A$ remains so over $\cl(A)$, and thus over $\acl(Ab)$, by Remark \ref{R:closure} (b). In particular, there is no  realisation of $p_0$ in $\acl(A\cup\{ b\})$ generic over $A$, so no realization of $p_0$ in $\acl(A\cup\{ b\})$ is $p$-independent over the field generated by $A$. Thus, the subfield $\acl(Ab)$ must have the same $p$-basis as $A$. 
	
	If $b$ is such that $\acl(Ab)$ has the same $p$-basis $B$ as $A=\acl(A)$,  every iterated $\lambda$-function of $b$ with respect to $B$ belongs to $\acl(A,b)\subset K^p[B]$, so $b$ belongs to $K^{p^m}[A]$ for every $m$ in $\N$ and hence $b$ lies in $\bigcap_{m\in \N} K^{p^m}[A]$. 
	
	For the last equality, suppose now that $b$ belongs to 
	$\bigcap_{m\in \N} K^{p^m}[A]$. In order to show that $b$ belongs to $\cl(A)$,  choose some $D\supset A$ and a generic element $g$ over $D$. By (the proof of) \cite[Lemma 4]{Sr86} we have that $\dcl(D\cup\{b\})$ is a subset of $\bigcap_{m\in \N} K^{p^m}[\acl(D)]$. If $g$ divides with $\dcl(D\cup\{b\})$ over $D$, then $g$ is not $p$-independent over $\dcl(D\cup\{b\})$, so \[ g\in K^p[\dcl(D\cup\{b\})] \subset K^p[\acl(D)], \] which is a contradiction since $g$ is generic over $D$. 
\end{proof} 

\begin{lemma}\label{L:bdd_scf}
	Every non-trivial automorphism of a $\kappa$-saturated model $K$ of the theory SCFP$_{p,\infty}$ moves $p_0$ maximally. In particular, the theory SCFP$_{p,\infty}$ satisfies  Property \ref{H:nobounded}. 
\end{lemma}
\begin{proof}
	We will be concise, for the proof is a straight-forward adaptation of
	\cite[Theorem 3.1]{BHMP17}.  Assume that the automorphism $\rho$ of the
	$\kappa$-saturated model $K$ does not move $p_0$ maximally so there is a 
	$\rho$-invariant subset $A=\acl(A)$ such that for every generic element $b$ over $A$, we have that $\rho(b)$ and $b$ are not independent over $A$. As $b$ is generic over $A$, the field $A(b)$ is separable and thus has some $p$-basis contained in $A\cup\{b\}$. Now, the element  $\rho(b)$ is not $p$-independent over $A(b)$, so $\rho(b)$ belongs to $K^p[A(b)]$. Hence, we deduce that $\rho(b)$ belongs to $L[b]$, where $L=K^p(A)$.

	Choose now two independent generic elements $b$ and $c$ over $A$. Our assumption implies that
	\[ \rho(b)\in L[b] \ , \ \rho(c)\in L[c] \ \& \ \rho(b+c) \in
	L[b+c].\]  Write
	$\rho(b)=\sum_{i=0}^{p-1}d_ib^i$, $\rho(c)=\sum_{i=0}^{p-1}e_ic^i$ and
	$\rho(b+c)=\sum_{i=0}^{p-1}f_i(b+c)^i$. We deduce that 
	\begin{align*} \sum_{i=0}^{p-1}d_ib^i+e_ic^i &=\rho(b)+\rho(c)= \rho(b+c) 
		=\sum_{i=0}^{p-1}f_i(b+c)^i= \\
		&=\sum_{i=0}^{p-1}f_i\sum_{j=0}^i {\binom{i}{j}}b^ic^{i-j}.
	\end{align*}
	Note that all $p$-monomials in $b$ and $c$ which appear in the last sum are linearly independent over $L$, for $b$ and $c$ are independent generic elements. In particular, we have that  $f_2=\ldots=f_{p-1}=0$, so 
	$f_0+f_1(b+c)=\rho(b+c)=(d_0+e_0)+d_1a+e_1b$ and thus \[ d_0+e_0=f_0 \text{ and } d_1=e_1=f_1,\] with $d_i=e_i=0$ for $i\geq 2$. For the generic element $b\cdot c$ over $A$ write  
	$\rho(b\cdot c)=\sum_{i=0}^{p-1}h_i(b\cdot c)^i$ with $h_i$ in $A$.  We deduce from the above that 
	\[\rho(b\cdot c)=(d_0+d_1b)(e_0+e_1 c)=d_0e_0+d_1e_0b+e_0d_1c+d_1e_1 b\cdot c,\] so we deduce that $h_i=0$ for $i\ge 2$ and \[ h_0=d_0e_0 \ ,\  h_1=d_1e_1 \ \text{ as well as } \  d_0e_1=d_1e_0=0.\] 
	Now, since $A$ is $\rho$-invariant, the generic element $\rho(b\cdot c)$ cannot lie in $L$, so $h_1\ne 0$. Hence $d_1\ne 0$, so $e_0=0$ and thus $h_0$ must be $0$, which means that for every pair $(b, c)$ of generic independent elements over $A$,  \[ \frac{\rho(b\cdot c)}{b\cdot c}=h_1 \in L.\] 
	Every generic element over $A$ can be written as a product of two generic independent elements over $A$, so $\rho(x)=\lambda_x x$ for every generic element $x$ over $A$ where $\lambda_x$ belongs to $L$. This yields the desired conclusion, since $x+1$ is also generic over $A$ and \[ \lambda_{x+1}x+ \lambda_{x+1}1=\lambda_{x+1}(x+1)=\rho(x+1)=\rho(x)+1=\lambda_x x +1,\] so $1=\lambda_{x+1}=\lambda_x$. Whence $\rho$ is the identity on $K$, for every element is a sum of two generic elements.
\end{proof}

Remark \ref{R:SCFP}, Lemmata \ref{L:SCF_closure} and \ref{L:bdd_scf} together with Lemma \ref{L:sat_tame}, Proposition \ref{P:prime_tame} and Corollary \ref{C:main} yield the following result, under the additional assumption that ${\rm cof}(\kappa)\geq \aleph_1$, since SCFP$_{p,\infty}$  is stable yet not superstable (see Remark \ref{R:Exist}).

\begin{cor}\label{C:SCF}
	For every subset $\paraset$ of parameters of cardinality strictly less than $\kappa$, every saturated model $\UU$ of size $\kappa$ of SCFP$_{p,\infty}$ as well as every $\kappa$-prime model over $\cl(\paraset)=\bigcap_{m\in \N} \UU^{p^m}[\paraset]$ are  $\kappa$-tame over $\paraset$ (Saturated models of SCFP$_{p,\infty}$ exist whenever $k\ge 2^{\aleph_0}$ whilst the existence and uniqueness of $\kappa$-prime models holds if ${\rm cof}(\kappa)\geq \aleph_1$). 
	
	For every $\kappa$-tame model $\UU$ of SCFP over $\paraset$, the group $\Aut(\UU/\cl(\paraset))$ is simple.
\end{cor}

Corollaries \ref{C:ACF},  \ref{C:ACFP} and \ref{C:SCF} as well as Propositions \ref{P:DCF} and \ref{P:ACFA} yield now the following result: 

\begin{theorem}\label{T:mainExamples}
	For each of the following countable theories of fields with operators:
	\begin{itemize}
		\item algebraically closed fields with the closure operator given by the
		field algebraic closure;
		\item differentially closed fields in characteristic $0$ with finitely
		many commuting  derivations with the closure operator given by the
		elements which are not differentially transcendental;
		\item  differential fields in characteristic $0$, maximal
		with the property of omitting a given strictly minimal
		type $X$; same closure as above;
		\item difference  closed fields  in characteristic $0$ with the closure
		operator given by the elements of transformal transcendence degree $0$;
		
		\item  proper pairs of algebraically
		closed fields $(K,E)$ with the closure operator \[ \cl(A)= {
			E(A)^\mathrm{alg}};\]  
		\item separably closed fields $K$ in characteristic $p$ and
		infinite imperfection degree with the closure
		operator \[\cl(A)=\bigcap\limits_{n\in\N} K^{p^n}[\acl(A)].\]            
	\end{itemize}
	The group of automorphisms of every uncountable model saturated in its
	uncountable cardinality $\kappa$ (if such models exist) fixing pointwise $\cl(\paraset)$ is simple, where 
	$\paraset$ is a subset of parameters of size strictly less than $\kappa$.

	More generally, for any of the above theories, given an uncountable cardinal
	$\kappa$ (with ${\rm cof}(\kappa)\geq \aleph_1$  in the last example) and a $\kappa$-prime model $\UU$ over $\cl(\paraset)$, where $\paraset$ is a subset of   size strictly less than $\kappa$,  the automorphism group 
	$\Aut(\UU/\cl(\paraset))$ is simple.\qed
\end{theorem}

We will finish this article  with some questions we did not attempt to solve, and propose a list of possible examples for further research. 

\begin{question}\label{Q:examples} The theory of separably closed 
	fields of
	positive characteristic $p$ and finite degree of imperfection $e$  eliminates imaginaries 
	\cite[Proposition 
	43]{fD88} in the language $
	\LL_\textrm{Rings}\cup\{c_1,\ldots, c_e\} 
	\cup\{\lambda_n(x)\}_{0\le n <
		p^e}$,  where $\{c_1,\ldots, c_e\}$ denotes a $p$-basis and the
	$\lambda$-functions are taken with respect to the monomials in that basis. This
	theory satisfies
	Properties \ref{H:single_gen}, \ref{H:closed_stat} and (WH)
	with respect to the closure operator given by 
	the
	unique generic type. We do not know whether this theory satisfies
	Property \ref{H:addgen}. By Remark \ref{R:SCF_noProp3}, the unique generic type is not regular if $e\ne 0$. 
	
	Providing an explicit algebraic description of the closure $\cl(A)$ in the finite degree of imperfection, even if we see such fields as fields equipped with (iterative) Hasse derivations \cite{mZ03}, seems difficult, so we do not know whether the Property \ref{H:addgen} holds. 
\end{question}

We have not attempted to list other natural examples of fields for which our methods could apply such as differentially closed fields of positive characteristic \cite{cW73}, DCFA \cite{Bu19}, ACFE$_0$ \cite{CH04} or $\mathcal D$-closed fields of characteristic $0$ equipped with $n$~free derivations \cite{MS14}.

\end{document}